\documentclass[11pt,a4paper]{article}
\usepackage{amsmath,amsfonts,amssymb,mathrsfs,amscd,comment}
\usepackage{graphicx}
\usepackage{xfrac}
\usepackage{subfigure}
\usepackage{pstricks}
\usepackage{pst-node}
\usepackage{xcolor}
\usepackage[english]{babel}
\usepackage{mathscinet}
\usepackage{tikz}
\usetikzlibrary{decorations.pathmorphing}
\usepackage{pgfplots}
\pgfplotsset{compat=1.16}

\newtheorem{theorem}{\bf Theorem}[section]
\newtheorem{lemma}[theorem]{\bf Lemma}

\newtheorem{definition}[theorem]{\bf Definition}
\newtheorem{remark}[theorem]{\bf Remark}
\newtheorem{proposition}[theorem]{\bf Proposition}
\newtheorem{example}[theorem]{\bf Example}

\newenvironment{proof}{\noindent {\sc Proof.}}{\hfill $\square$}

\newcommand{\R}{\mathbb{R}}

\def \rnn {{\mathbb {R}}^{N+1}}

\def \LO {\mathscr{L}_0} %{{\mathcal {L}}}
\def \G {{\Gamma}}
\def \g {{\gamma}}

\def \Q {{\mathcal{Q}}}

\def \OO {{\mathbb{O}}}

\def \I {{\mathbb{I}}}
\def \K {\mathcal{K}}
\def \L {\mathscr{L}}
\def \H {\mathscr{H}}
\def \Heat {\mathcal{H}}

\def \g {{\gamma}}
\def \d {{\delta}}

\def \k {{\kappa}}
\def \r {{\varrho}}

\def \t {{\tau}}

\def \x {{\xi}}

\def \z {{\zeta}}

\def \phi {{\varphi}}

\def \G {{\Gamma}}
\def \O {{\Omega}}

\def \GG {{\mathfrak{g}}}

\def \div {{\text{\rm div}}}

\def \diag {{\text{\rm diag}}}

\def\p{\partial}
\def \tilde {\widetilde}

\def \AS {{ \mathscr{A}_{z_0} }}

\newcommand{\scp}[1]{\langle #1 \rangle}
\newcommand{\mres}{\mathbin{\vrule height 1.6ex depth 0pt width 0.13ex\vrule height 0.13ex depth 0pt width 1.3ex}}

\usepackage{mathtools}

\definecolor{brickred}{rgb}{0.8, 0.25, 0.33}
\definecolor{forestgreen}{rgb}{0.13, 0.55, 0.13}
\newrgbcolor{sergioblue}{ .02 0 .7}
\newenvironment{sergiorev}{\color{sergioblue}}{\color{black}}
\newcommand{\bsr}{\begin{sergiorev}}
\newcommand{\esr}{\end{sergiorev}}
\newenvironment{sergiorevii}{\color{forestgreen}}{\color{black}}
\newcommand{\bsrr}{\begin{sergiorevii}}
\newcommand{\esrr}{\end{sergiorevii}}

\newenvironment{diegorev}{\color{brickred}}
{\color{black}}
\newcommand{\bdr}{\begin{diegorev}}
\newcommand{\edr}{\end{diegorev}}

\begin{document}
\title{Mean value formulas for classical solutions to subelliptic evolution equations in stratified Lie groups}
\author{\sc{Diego Pallara}
\thanks{Dipartimento di Matematica e Fisica ``Ennio De Giorgi'', Universit\`{a} del Salento and INFN, Sezione di Lecce, Ex Collegio Fiorini - Via per Arnesano - Lecce (Italy). E-mail: 	diego.pallara@unisalento.it} \ 
\sc{Sergio Polidoro}
\thanks{Dipartimento di Scienze Fisiche, Informatiche e Matematiche, Universit\`{a} degli Studi di Modena e Reggio Emilia, Via Campi 213/b, 41125 Modena (Italy). E-mail: 	sergio.polidoro@unimore.it}
}
	
\date{ }
	
\maketitle

\bigskip

\begin{abstract}
We prove mean value formulas for classical solutions to second order linear differential equations in the form
 $$
 \p_t u = \sum_{i,j=1}^m X_i (a_{ij} X_j u) + X_0 u + f,
 $$
where $A = (a_{ij})_{i,j=1, \dots,m}$ is a bounded, symmetric and uniformly positive matrix with $C^1$ coefficients under the assumption that the operator $\displaystyle\sum_{j=1}^m X_j^2 + X_0 - \p_t$ is hypoelliptic and the vector fields $X_1, \dots, X_m$ and $X_{m+1} :=X_0 - \p_t$ are invariant with respect to a suitable homogeneous Lie group. Our results apply e.g. to degenerate Kolmogorov operators and parabolic equations on Carnot groups $\displaystyle \p_t u = \sum_{i,j=1}^m X_i (a_{ij} X_j u) + f$.
\end{abstract}
	
\setcounter{equation}{0} 

\section{Introduction}\label{secIntro}

The aim of this paper is to prove mean value formulas for degenerate  second order partial differential equations in the form
\begin{equation}\label{e1}
     \L u : = \sum_{i,j=1}^m X_i \left( a_{ij} X_j u \right) + X_0 u 
     + \sum_{j=1}^m b_{j} X_j u + c u - \p_t u = f,
\end{equation}
in some open set $\Omega \subset \R^{N+1}$. In the following, $z=(x,t)=(x_{1},\dots,x_{N},t)$ denotes the point in $\R^{N+1}$, $1 \le m \le N$ and the $X_{j}$'s in \eqref{e1} are smooth vector fields on $\R^{N}$, i.e.,
\begin{equation}\label{defX}
     X_{j}(x)=\sum_{k=1}^{N} \varphi_{k}^{j}(x)\p_{x_{k}}, \qquad j=0,\dots,m,
\end{equation}
$\varphi_{k}^{j}$ being $C^{\infty}$ functions. Starting from the vector fields $X_j$, the operator $\L$ in \eqref{e1} is written through the $m \times m$ symmetric matrix $A(z) = \left( a_{ij}(z) \right)_{i,j = 1, \dots, m}$, which has continuous entries and satisfies the usual ellipticity condition: there exists a constant $\Lambda \geq 1$ such that
\begin{equation}\label{elliptic}
  \Lambda^{-1}|\x|^{2}\le \sum_{i,j=1}^{m} a_{ij}(z)\x_{i}\x_{j} \le \Lambda |\x|^{2},
\end{equation}
for every $z \in \R^{N+1}$ and $\x\in \R^{m}$. To simplify the notation in the sequel, we still denote by $A$ the $(m+1)\times (m+1)$ matrix with entries $a_{ij}$ for $i,j=1,\ldots, m$ and $a_{(m+1)\, j}= a_{j\, (m+1)}=0$, for $j= 1, \dots, m+1$. We assume that the coefficients of the vector $b(z) = (b_1(z), \dots, b_m(z))$ and the functions $c$ and $f$ are bounded and continuous. As we are dealing with classical solutions, we also assume that $X_ia_{ij}$ and $X_jb_{j}$ are bounded continuous functions for $i,j\in\{1,\ldots, m\}$. The reason why we write the operator $\L$ in its divergence form is that we need to consider the fundamental solution $\Gamma^*$ to the adjoint equation $\L^* v = 0$. In the sequel we denote 
\begin{equation}\label{eY}
  X = \left( X_{1}, \dots, X_{m} \right), \qquad X_{m+1}=X_0-\p_{t}.
  % ,\qquad   \div_X F = \sum_{j=1}^m X_j F_j,  
\end{equation}
and we let
\begin{equation} \label{e-Lie}
    \GG = {\rm Lie} (X_1, \ldots, X_m, X_{m+1})
\end{equation}
be the Lie algebra generated by $X_1, \ldots, X_m, X_{m+1}$. The main assumptions on the operator $\L$ are listed below. 

\begin{description}
\item[{\rm [H.1]}] The vector fields $X_1,\ldots,X_{m+1}$ satisfy the H\"{o}rmander rank condition
\begin{equation} \label{e-hrc}
    {\rm rank} \, \GG (z) = N+1 \qquad \text{for every} \quad z \in \R^{N+1}.
\end{equation}
  \item[{\rm [H.2]}] there exists a homogeneous Lie group
  $\mathbb{G}=\left(\rnn,\circ, \d_{\lambda}\right)$ such that
\begin{description}
  \item[{\it i)}] $X_{1},\dots,X_{m},X_{m+1}$ are 
  left-translation invariant on   $\mathbb{G}$;
  \item[{\it ii)}] $X_{1},\dots,X_{m},X_{m+1}$ are $\d_{\lambda}$-homogeneous of degree one; 
\end{description}
  \item[{\rm [H.3]}] There exists a fundamental solution $\Gamma^*$ for the adjoint operator $\L^*$:
  \begin{equation}\label{e1-ad}
     \L^* u : = \sum_{i,j=1}^m X_i \left( a_{ij} X_j u \right) - X_0 u 
     - \sum_{j=1}^m b_{j} X_j u + \Big( c - \sum_{j=1}^m X_j b_{j} \Big) u + \p_t u,
\end{equation}
as stated in Definition \ref{def-fund-sol}, having the properties \eqref{eq-claim-1} and \eqref {eq-claim-2}.
\end{description}

Let us briefly comment on our hypotheses. We first recall that the H\"{o}rmander condition {\rm [H.1]} implies that the operator 
\begin{equation}\label{e0}
  \LO:= \sum_{k=1}^{m}X_{k}^{2}+X_{m+1}
\end{equation}
is hypoelliptic. This means that every distributional solutiont $u$ to $\LO u = f$ in some open set $\Omega \subset \R^{N+1}$ belongs to $C^\infty(\Omega)$ and is a classical solution to $\LO u = f$ whenever $f$ is $C^\infty(\Omega)$. Note that $\LO$ is an operator of the form \eqref{e1} if we choose $A$ to be the $m \times m$ identity matrix, and $b=0, c = 0$. From this point of view, in the setting of the degenerate operators, $\LO$ plays the role of the heat operator in the family of the uniformly parabolic operators.

In Section \ref{secNotation} we recall the notation of homogeneous Lie groups $\mathbb{G}=\left(\rnn,\circ, \d_{\lambda}\right)$ and we give necessary and sufficient conditions on the vector fields $X_{1},\dots,X_{m},X_{m+1}$ for the validity of  condition {\rm [H.2]}. In Section \ref{secNotation} we also declare the properties of the \emph{fundamental solution} we need in this note, so that the meaning of condition {\rm [H.3]} will be clarified. 

Mean value formulas are a very classical tool for the treatment of harmonic functions. The first extenstions to the heat operator are due to Pini \cite{Pini1951} and to Watson \cite{Watson1973}. Mean value formulas for uniformly parabolic operators with $C^\infty$ smooth coefficients have been proved by Fabes and Garofalo \cite{FabesGarofalo}, and by Garofalo and Lanconelli \cite{GarofaloLanconelli-1989}. More recently, mean value formulas have ben proved in large generality by Cupini and Lanconelli in \cite{CupiniLanconelli2021}, who still consider operators with smooth coefficients. In the recent articles \cite{MalPalPol} on uniformly parabolic equations and \cite{PalPol}, on degenerate elliptic equations, Malagoli and the authors consider differential operators whose coefficients belong to a suitable space of functions $C^{1+\alpha}$ and $C^{1+\alpha}_{\mathbb G}$, respectively. This reduction of regularity, from $C^\infty$ to $C^{1+\alpha}$, reflects in weaker regularity of the fundamental solution. As a consequence, the integration by parts on its level sets, which is a crucial point in the proof and for $C^\infty$ coefficients relies on Sard's theorem, becomes a delicate issue. In the case of \emph{uniformly parabolic} operators with $C^1$ coefficients, and then $C^1$ fundamental solution, in \cite{MalPalPol} it is shown that this difficulty can be overcome in two ways, either by using a divergence theorem valid for \emph{almost $C^1$-regular boundaries}, or using De Giorgi's theory of perimeters. In \cite{PalPol} and in the present case of \emph{degenerate} operators with $C^1_{\mathbb G}$ coefficients the \emph{almost $C^1$-regularity} of the level sets of the fundamental solution is not guaranteed: indeed, it is not $C^1$ in the Euclidean sense. So, we are led to rely on the theory of functions with bounded variation and sets with finite perimeter in stratified groups, see Section \ref{secBV}, which provides us with the relevant formulation of the divergence theorem. In Section \ref{secProofs} we prove our main results, i.e., the mean value formulas and, as an important consequence, the strong maximum principle. Finally, in Section \ref{SecExamples} we show that our assumptions are verified in several important cases, such as degenerate Kolmogorov operators and parabolic operators on Carnot groups. 

We conclude this introduction with a remark on the the homogeneity property {\rm [H.2]} {\it ii)}. The layers in the Lie algebra $\mathfrak g$ have different degrees of homogeneity and these are at some extent arbitrary. Assuming for simplicity $X_0=0$, the commonest choice when dealing with regularity theory is the {\em parabolic scaling} $\delta_\lambda: (x,0)=(\lambda x,0)$ for $x$ in the first layer, and $\delta_\lambda: (0,t) \mapsto (0,\lambda^2 t)$  on the time variable. But, as already done in \cite{MalPalPol} for uniformly parabolic operators in $\R^{N+1}$, the condition {\rm [H.2]} {\it ii)} requires that $\delta_\lambda (x,t) = (\lambda x,\lambda t)$ and we do the same here, with $x$ in the first layer. Indeed, the very important property of $\Gamma^*$ encoded in the equality 
$$
  u(\x,\t) =\int_{\R^{N}}\Gamma^*(\x,\t;x, t)\phi(x)d x 
$$
cannot be used in connection with the divergence formula when the boundary integral is computed with respect to the Hausdorff measure, if this last is defined through the distance generated by the parabolic scaling. In fact, the surfaces $\{t=constant\}$ would have codimension 2 and therefore they would be negligible. 

\begin{comment}
\bsr We stress that, even though the parabolic scaling is the natural one in the regularity theory for parabolic equations, it \esr is inconvenient in the representation of the perimeter of the super-level sets of the fundamental solution in terms of the Hausdorff measure. 
% (\bsr Ho tolto: which is a fundamental tool in the proof of our mean value formulas. \esr)
Indeed, if we consider the Hausdorff measure build by parabolic cylinders $Q_r(x_0,t_0) = B_r(x_0) \times ]t_0-r^2, t_0 + r^2[$, being $B_r(x_0)$ the Euclidean ball with center at $x_0$ and radius $r$, then the homogeneous dimension of $\R^{N+1}$ is $N+2$, and the hyperplanes $\{ t = t_0 \}$ have dimension $N$. In particular they are negligible sets for this %(\bsr Ho tolto: the relevant Hausdorff measure, then\esr) 
measure, so that we cannot use it to represent the integral appearing in 
$$
  u(\x,\t) =\int_{\R^{N}}\Gamma^*(\x,\t;x, t)\phi(x)d x, 
$$ $\delta_\lambda(t)=\lambda^2t$
which is a very important property of $\Gamma^*$ we need in the proof of our main results. For this reason we deal with a dilation group $\big( \delta_\lambda \big)_{\lambda >0}$ satisfying {\rm [H.2]} {\it ii)}. 
\end{comment}
\medskip

\noindent {\bf Acknowledgments.} The authors are members of Gruppo Naziona\-le per l’Analisi Matematica, la Probabilità e le loro Applicazioni (GNAMPA) of the Istituto Nazionale di Alta Matematica (INdAM). We thank Nicola Garofalo and Ermanno Lanconelli for their interest in our work and Valentino Magnani, Francesco Serra Cassano, Francesca Tripaldi and Davide Vittone for several useful discussions. 

\setcounter{equation}{0} 

\section{Preliminaries and the Main Results}\label{secNotation}

In this section we describe the group structure related to our differential operators and the properties that we require to the fundamental solutions. After that, we collect our hypotheses and state our main results. Further preliminaries concerning $BV$ functions and sets with finite perimeter are presented in Section \ref{secBV}. 

Let us specify the meaning of classical solution to the equation $\L u = f$ in some open set $\Omega \subset \R^{N+1}$. With this aim, we first recall the notion of \emph{Lie derivative}. For any $z_0 \in \O$ and $j=1, \dots, m+1$ we consider a path $\gamma$ defined in a neighborhood $I$ of the origin, such that
\begin{equation*}
 \gamma'(s) = X_j(\gamma(s)), \qquad \gamma(0)= z_0.
\end{equation*} \label{eq-Lie derivative}
Then the Lie derivative $X_j u(z_0)$ of $u$ at $z_0$ is
\begin{equation}
 X_j u(z_0) := \frac{d}{d s}u(\gamma(s))_{\mid s = 0}.
\end{equation}
We say that $u$ is a classical solution to $\L u = f$ in some open set $\Omega \subset \R^{N+1}$ if the Lie derivatives $X_i X_j u, i,j=1, \dots, m$ and $X_{m+1} u$ are defined as continuous functions and the differential equation is satisfied at every point of $\Omega$. The meaning of the  equation $\L^* u = g$ is analogous.

Next subsections contain some facts about the notions of Lie group and fundamental solution, that are likely known to the expert readers. We first recall some notation and results, then we state the main goal of this paper. 

\subsection{Homogeneous Lie groups}\label{SubsGroup}

A Lie group $\mathbb{G}=\left(\rnn,\circ\right)$ is said {\it homogeneous} if a family of dilations $\left(\d_{\lambda}\right)_{\lambda>0}$ exists on $\mathbb{G}$ and it is an automorphism of the group:
\begin{equation*}
\d_\lambda(z \circ \z) = \left(\d_\lambda z\right) \circ \left( \d_\lambda \z\right), \quad \text{for all} \ z, \z
\in \rnn \ \text{and} \ \lambda >0.
\end{equation*}
The assumptions [H.1] and [H.2] induce a direct sum decomposition of $\GG$
\begin{equation}\label{e-Oplus-bis}
    \GG = W_{1}\oplus\cdots\oplus W_{\mu},
\end{equation}
where 
\begin{equation*}
\begin{split}
    W_1 = & \text{span} \big\{X_1, \ldots X_m, X_{m+1} \big\}, \\
    W_k = & \text{span} \big\{[X_i, X_j]\mid X_i \in W_1, X_j \in W_{k-1} \big\}, \qquad k=2, \ldots, \mu.
\end{split}
\end{equation*}
Moreover $[X_i, X_j] = 0$ whenever $X_i \in W_\mu$, and  $X_j \in W_1$. In the sequel we denote by $m_j$ the dimension of $W_j$, for $j=1,\dots, \mu$. If we represent the dilation $\delta_\lambda$ on $\R^{N+1}$ by the following matrix
\begin{equation} \label{e-delta}
  \delta_\lambda = \diag ( \lambda \I_{m_1} , \lambda^2 \I_{m_2}, \ldots, \lambda^{\mu} \I_{m_\mu}),
\end{equation}
we call \emph{homogeneous dimension} of $\mathbb{G}$ the integer $\mathcal{Q} = m_1 + 2 \mu_2 + \cdots + \mu m_\mu$ and we have 
\begin{equation}\label{e-Q}
    \det \delta_\lambda = \lambda^{\mathcal{Q}}.
\end{equation}
We endow each fiber of the first layer $W_1$ with an inner product $\langle\cdot,\cdot\rangle_z$ (and the associated norm $|\cdot|_z$) that makes $X_1(z),\ldots,X_{m+1}(z)$ an orthonormal frame, see \cite{fraserser2}. 

We next give some general comments about the differential operators on homegeneous Lie groups considered in this note. An important consequence of the homogeneity of the Lie group is the \emph{pyramid-shaped} structure of the coefficients of the vector fields $X_j$'s. In the following we write  $x=x^{(1)}+x^{(2)}+\cdots+x^{(k)}$ with $x^{(j)}\in W_{j}$, and we use the notation in \eqref{defX}. As a consequence of the homogeneity of the differential operators $X_1, \ldots, X_{m+1}$ we have that the coefficients $\varphi^j_1, \ldots, \varphi^j_m$ are constant and the coefficients $\varphi^j_{m+1}, \ldots, \varphi^j_{m+m_2}$ are linear functions of the variable $x^{(1)}$.  In general, if $k$ is such that $x_k \in W_i$, then the coefficient $\varphi^j_{k}$ is a polynomial function of the variable $x^{(1)}, \ldots, x^{(i-1)}$. This fact plainly implies that
\begin{equation}\label{e-adj}
  X_j^* = - X_j, \qquad j=1, \dots, m+1. 
\end{equation} 
In particular, the adjoint operator $\L^*$ acts on sufficiently smooth functions as follows
\begin{equation}\label{e1*}
     \L^* v : = \sum_{i,j=1}^m X_i \left( a_{ij} X_j v \right) - X_0 v 
     - \sum_{j=1}^m b_{j} X_j v + 
     \Bigl(c-\sum_{j=1}^m X_jb_j\Bigr) v + \partial_t v.
\end{equation}
For the same reason, $X_1, \dots, X_{m+1}$ are \emph{complete vector fields}, that is the integral curve of $X_j$ is defined on the whole of $\R$ for $j=1, \dots, m+1$. 

We next discuss the problem of the existence of a Lie group $\mathbb{G}$ as required in condition {\rm [H.2]}. In the Examples \ref{ex-1} and \ref{ex-2} both the Lie group and the differential operators are given. In general, when a stratified Lie group is given, a standard procedure provides us with a family of left-invariant vector fields satisfying the H\"ormander condition {\rm [H.1]}. 

Vice versa, suppose that a collection of vector fields $X_1, \dots, X_{m+1}$ satisfying the H\"ormander condition {\rm [H.1]} is given. Then, under some furhter assumptions, it is possible to build a stratified Lie group such that the condition {\rm [H.2]} is satisfied. Specifically the following result holds true. Suppose that ${\rm Lie} (X_1, \ldots, X_m, X_{m+1})$ has dimension $N+1$ and that every vector field $X \in {\rm Lie} (X_1, \ldots, X_m, X_{m+1})$ is complete. Then there exists a Lie group $\mathbb{G} \left( \R^{N+1}, \circ \right)$ such that the Lie algebra of $\mathbb{G}$ agrees with ${\rm Lie} (X_1, \ldots, X_m, X_{m+1})$. The abstract version of this result is known as \emph{Third Fundamental Theorem of Lie}, see \cite[Theorem 3.15.1]{Varadarajan}. We refer to Theorem 1.1 in the article \cite{BonfLancCPAA} by Bonfiglioli and Lanconelli for the explicit construction of this Lie group on $\R^{N+1}$ in the case of vector fields with analytical coefficiens, while $C^\infty$ vector fields are considered by Biagi and Bonfiglioli in the more recent articles \cite{BiagiBonfiglioli} and \cite{Bonfiglioli}. 

We next introduce the distance we use to define the Hausdorff measure related to the notion of perimeter on which the divergence formula we need is based. From \cite[Theorem 5.1]{fraserser3} we know that there are constants $\varepsilon_j\in ]0,1],\ j=1,\ldots,\mu$, with $\varepsilon_1=1$, such that the function
\begin{equation}\label{defnorma_infty}
z \mapsto \|z \|_\infty = \max_{j=1,\ldots,\mu}\{\varepsilon_j|z_j|^{1/j}\},
\end{equation}
where $|z_j|$ denotes the Euclidean norm of the vector $z_j\in W_j$, defines a norm and as a consequence the distance
\begin{equation}\label{defd_infty}
d_\infty(z,w)=\|w^{-1}\circ z\|_\infty.
\end{equation}
We notice that $d_\infty$ is equivalent to the Carnot-Carath\'{e}odory distance (see \eqref{DefCCdist} below) and that for every compact set $K \subset \R^{N+1}$ there exist two positive constants $c_K^-$ and $c_K^+$, such that
\begin{equation}\label{e-dist-eucl}
  c_K^- |z-w| \le d_\infty(z,w) \le c_K^+ |z-w|^\frac{1}{\mu}, \qquad \text{for all} \ z,w \in K.
\end{equation}
The invariance properties
\begin{equation} \label{e-dist-INV}
   d_\infty (\zeta \circ z, \zeta \circ w) =  d_\infty (z,w), \quad
   d_\infty (\delta_\lambda z, \delta_\lambda w) = \lambda \, d_\infty (z,w),
\end{equation}
hold for every $z, w, \zeta$ in $\R^{N+1}$ and for every positive $\lambda$, see again \cite{fraserser3}. We next recall the notion of H\"older continuous functions on Lie groups. For $\alpha \in ]0,1]$, we say that a function $u$ defined on $\O$ is $\alpha$-H\"older continuous, and we write $u \in C_{\mathbb G}^\alpha(\Omega)$, if there exists a positive constant $M$ such that
\begin{equation} \label{eq-C-alpha}
 |u(z)-u(w)| \le M d_\infty (z,w)^\alpha, \quad \text{for every} \ z,w \in \Omega.
\end{equation}
Let us come to some geometric measure theoretical issues. For every $\nu\in W_1$, denote by $\nu^\bot$ the codimension 1 subspace of $W_1$ orthogonal to $\nu$ and introduce the hyperplane $\mathcal{N} =\nu^\bot\oplus W_2\oplus \cdots \oplus W_\mu$ in $\R^{N+1}$ and the constant
\begin{equation}\label{deftheta}
\theta= \{\H_e^{N}(B(0,1)\cap \mathcal{N})\} 
%\max_{z\in B(0,1)}\{\H_e^{N}(B(z,1)\cap \mathcal{N})\},
\end{equation}
where $B(0,1)=\{z\in\R^{N+1}|\ \|z\|_\infty\leq 1\}$ and $\H_e^{N}$ is the $N$-dimensional euclidean Hausdorff measure, see e.g. \cite[Section 2.8]{AmbFusPal00Fun}. The constant $\theta$ is introduced in \cite{MagnaniCalcVar}, is called spherical factor and is denoted $\omega_{{\mathbb G},\mathcal{Q}-1}$ there. Moreover, it is independent of $\nu$ because $d_\infty$ is \emph{vertically symmetric} according to Definition 6.1 in \cite{MagnaniIUMJ}, see Remark 6.2, Theorem 6.3 and Theorem 5.2 in \cite{MagnaniIUMJ}. We then recall the definition of the {\em spherical (${\mathcal Q}-1$)-dimensional Hausdorff measure ${\mathcal S}_{\mathbb G}^{{\mathcal Q}-1}$} of a Borel set $E$:  
\begin{equation*}
{\mathcal S}_{\mathbb G}^{{\mathcal Q}-1}(E) =\lim_{r\downarrow 0}
\inf\left\{
\sum_{i=0}^\infty \frac{\theta}{2^{{\mathcal Q}-1}}
({\rm diam}_{\mathbb G}(B_i))^{{\mathcal Q}-1}: B_i\text{  balls },\ E\subset \bigcup_{i=0}^\infty B_i,
\ {\rm diam}_{\mathbb G}\,(B_i)\leq r \right\},
\end{equation*}
where $\theta$ is the constant in \eqref{deftheta} and ${\rm diam}_{\mathbb G}(B):=\sup_{z,\zeta\in B}d_\infty(z,\zeta)$. Notice that we have normalized the measure ${\mathcal S}_{\mathbb G}^{{\mathcal Q}-1}$ so that ${\mathcal S}_{\mathbb G}^{{\mathcal Q}-1}(B(0,1)\cap {\mathcal N})=\theta$. 

\subsection{The fundamental solution}\label{SubsFundSol}

In this subsection we give a precise definition of {\em fundamental solution} $\Gamma^*$ for the operator ${\mathscr L}^*$ in  \eqref{e1-ad}, we fix the notation for its superlevel sets  and we list some further assumptions on $\Gamma^*$. 

\begin{definition} \label{def-fund-sol} We say that a function $\G^* = \G^*(\z; z)$ defined for every $(\z; z) = (\x,\t;x,t) \in \R^{N+1} \times \R^{N+1}$ with $t> \tau$ is a fundamental solution to $\L^* v = 0$ if it satisfies the following conditions for every $z = (x, t) \in \R^{N+1}$.
\begin{enumerate}
\item The function $\G^*( \, \cdot \, ; z)$ is a classical solution to the equation $\L^* \, \G^*(\cdot \,; z) = 0$ in $\R^{N} \times ]- \infty, t[$;
\item the function $\G^*( \, \cdot \,; z)$ belongs to $L^1(K)$ for every bounded measurable set $K \subset \R^{N} \times ]-\infty, t[$. Moreover, for every $(\x,\t) \in \R^{N} \times ]- \infty, t[$ and $\phi \in C_c(\R^N)$, the function
$$
  u(\x,\t):=\int_{\R^{N}}\Gamma^*(\x,\t;x, t)\phi(x)d x 
$$
is well defined and satisfies
$$
    \lim_{(\x,\t) \to (x,t)}u(\x,\t) = \varphi (x) \quad \text{for every} \quad x \in \R^N.
$$
\end{enumerate}
\end{definition}
We introduce some further notation in order to state our last assumption on $\Gamma^*$. Recall that $\mathcal{Q}$ denotes the homogeneus dimension of the Lie group $\mathbb{G}$. For every $z_0=(x_0, t_0) \in \R^{N+1}$ and for every $r>0$, we set
\begin{equation} \label{e-Psi}
\begin{split}
 \psi_r(z_0) & := \left\{ z \in \R^{N} \times ]- \infty, t_0[ \mid \Gamma^*(z; z_0) = \tfrac{1}{r} \right\}, \\
 \Omega_r(z_0) & := \left\{ z \in \R^{N} \times ]- \infty, t_0[ \mid \Gamma^*(z;z_0) > \tfrac{1}{r} \right\},
\end{split}
\end{equation}
and we call $\psi_r(z_0)$ and $\Omega_r(z_0)$ respectively the \emph{sphere} and the \emph{ball} with radius $r$ and \emph{center} $z_0$. As in the uniformly parabolic setting, here $z_0$ belongs to the topological boundary of $\Omega_r(z_0)$. We finally set
\begin{equation} \label{eq-Irepsilon}
 \mathcal{I}_{r, \varepsilon} (z_0) := \left\{ x \in \R^N \mid \Gamma^*(x,t_0-\varepsilon;z_0) > \tfrac{1}{r} \right\}.
\end{equation}
Note that $\mathcal{I}_{r, \varepsilon} (z_0) \ne \emptyset$ only for sufficiently small positive $\varepsilon$. We also rely on the following properties of the fundamental solution
\begin{itemize}
 \item For every $z_0 \in \R^{N+1}$ there exists $r_0>0$ such that the set
\begin{equation} \label{eq-claim-1}
 \Omega_r(z_0) \quad \text{is bounded for every} \ r < r_0;
\end{equation}
 \item  it holds
\begin{equation} \label{eq-claim-2}
\lim_{\varepsilon \to 0} {\mathscr H}_e^{N} \left( \mathcal{I}_{r, \varepsilon} (z_0) \right) = 0, \qquad
 \lim_{\varepsilon \to 0} \int_{\R^N \backslash \mathcal{I}_{r, \varepsilon} (z_0)} \Gamma^*(x,t_0 - \varepsilon; z_0) d x = 0.
\end{equation}
\end{itemize}
Note that this assumptions is analogous to the \emph{pointwise vanishing integral condition} stated in \cite{CupiniLanconelli2021} as Property $(H(z_0, \varrho))$. 
% \bdr We also note that in \cite{MalPalPol} the set $\Omega_r(z_0)$ is defined as $\left\{ \Gamma^*(z;z_0) > \tfrac{1}{r^N} \right\}$. Of course, considering $\tfrac{1}{r^{N}}$ instead of $\tfrac{1}{r}$ does not change the matter of facts, we choose this option as it preserves the analogy with the classical case of the Laplace equation. \edr

\subsection{Statement of the main results}\label{MainRes}

We define the gradient $\nabla_{\mathbb G}$ and the divergence ${\rm div}_{\mathbb G}$ as follows. We agree to identify a \emph{section} $F=\sum_{j=1}^{m+1} F_j X_j$ with its canonical coodinates $F= \left(F_1, \dots, F_m, F_{m+1}\right)$. With this agreement, we denote the gradient of $f\in C^1_{\mathbb G}(\R^n)$ and the divergence of $F\in C^1_{\mathbb G}(\R^n,\R^{m+1})$ by
\begin{equation}\label{defDivGrad}
\nabla_{\mathbb G} f: =\sum_{j=1}^{m+1} (X_j f) X_j \quad{\rm and} \quad  
{\rm div}_{\mathbb G}F := - \sum_{j=1}^{m+1} X_j^*F_j = \sum_{j=1}^{m+1} X_jF_j.
\end{equation}
Moreover 
\begin{equation} \label{e-kernels-QF}
	\langle A(z) \nabla_{\mathbb G} f(z), \nabla_{\mathbb G} f(z) \rangle_z := 
	\sum_{j,k=1}^{m} a_{jk}(z) X_j f(z) X_k f(z),
\end{equation}
for any $f\in C^1_{\mathbb G}(\R^n)$. In the following, $z_0\in \R^{N+1}$ is fixed and $\nabla_{\mathbb G}\Gamma^*(z;z_0)$ denotes the gradient with respect to the variable $z$.  With this notation, we set \\
%\bdr aggiustare le costanti e gli esponenti dopo aver controllato Lemma 3.9 e Prop. 4.1. \edr \\
\begin{equation} \label{e-kernels}
\begin{split}
	K_{\mathbb G} (z_0; z) & := 
	\frac{\langle A(z) \nabla_{\mathbb G} \Gamma^*(z;z_0), \nabla_{\mathbb G} \Gamma^*(z;z_0) \rangle_z }
	{|\nabla_{{\mathbb G}}\Gamma^*(z_0;z)|_z},\\
	M_{\mathbb G}(z_0; z) & :=  
	\frac{\langle A(z) \nabla_{\mathbb G} \Gamma^*(z;z_0), \nabla_{\mathbb G} \Gamma^*(z;z_0) \rangle_z }
	{\Gamma^*(z;z_0)^2}.
\end{split}
\end{equation}
Note that, according to \eqref{defDivGrad} and the definition of the matrix $A$, $\nabla_{\mathbb G}\Gamma^*(z; z_0) \in \R^{m+1}$, while only its first $m$ components  appear in \eqref{e-kernels-QF}. Moreover, we agree to set $K (z_0; z) = 0$ whenever $\nabla_{\mathbb G}\Gamma^*(z;z_0)=0$. The first achievements of this note are the following mean value formulas.  

\begin{theorem} \label{th-1}
Let $\L$ be a differential operator in the form \eqref{e1}, satisfying the ellipticity condition \eqref{elliptic} and $a_{ij}= a_{ji}$ for $i,j= 1, \dots,m$, as well as hypotheses {\rm [H.1]}, {\rm [H.2]} and {\rm [H.3]}. Suppose that the coefficients $a_{ij},b_j, c, X_i a_{ij},X_j b_j$ are continuous, and that \eqref{eq-claim-1} and \eqref{eq-claim-2} hold.
 
Let $\Omega$ be an open subset of $\R^{N+1}$, $f\in C(\Omega)$ and let $u$ be a classical solution to $\L u = f$ in $\Omega$. Then, for every $z_0 \in \Omega$ and for almost every $r \in ]0,r_0]$ such that $\overline{\Omega_r(z_0)} \subset \Omega$ we have 
\begin{align*} %\label{e-meanvalue}
	u(z_0) =& \int_{\psi_r(z_0)} K_{\mathbb G} (z_0; z) u(z) \, d{\mathcal S}_{\mathbb G}^{{\Q}-1}(z) 
	+  \int_{\Omega_r(z_0)} f (z) \left( \tfrac{1}{r} - \Gamma^*(z;z_0) \right)\, dz  \\ 
	& +  \frac{1}{r} \int_{\Omega_r(z_0)} \big( \div_{\mathbb G} b(z) - c(z) \big) u(z) \, dz, \\ 
    u(z_0) =& \frac{1}{r} \int_{\Omega_r(z_0)} M_{\mathbb G} (z_0; z) u(z) \, dz 
	+ \frac{1}{r}  \int_0^{r} \Bigl(
	\int_{\Omega_\r (z_0)} f (z) \bigl( \tfrac{1}{\r} - \Gamma^*(z;z_0) \bigr) dz \Bigr)\, d \r 
	\\
    &+ \frac{1}{r} \int_0^{r}  \tfrac{1}{\r} \Bigl(
	\int_{\Omega_\r (z_0)} \big( \div_{\mathbb G} b(z) - c(z) \big) u(z) \, dz \Bigr) \, d \r. 
\end{align*} 
The second statement holds for \emph{every} $r\in ]0,r_0]$ such that ${\Omega_r(z_0)} \subset \Omega$. 
\end{theorem}

\begin{remark}\label{esponente}
{\rm An inspection of the proof of the second assertion in Theorem \ref{th-1} shows that we can modify the kernel $M_{\mathbb G}$ by raising $\Gamma^*$ to any exponent $\alpha>1$. Indeed, setting 
\[
M_{{\mathbb G},\alpha}(z_0; z) :=  
	\frac{\langle A(z) \nabla_{\mathbb G} \Gamma^*(z;z_0), \nabla_{\mathbb G} \Gamma^*(z;z_0) \rangle_z }
	{\Gamma^*(z;z_0)^\alpha} 
\]
we find
\begin{align*}
u(z_0) =& \frac{\alpha -1}{r^{\alpha -1}} \int_{\Omega_r(z_0)} M_{{\mathbb G},\alpha} (z_0; z) u(z) \, dz 
	+ \frac{\alpha-1}{r^{\alpha-1}} \int_0^{r} \Bigl(\r^{\alpha-2}
	\int_{\Omega_\r(z_0)}f(z) \bigl( \tfrac{1}{\r} - \Gamma^*(z;z_0)\bigr)dz \Bigr)\, d \r 
	\\
    &+ \frac{\alpha-1}{r^{\alpha-1}}\int_0^{r}\Bigl(\r^{\alpha-3}
	\int_{\Omega_\r (z_0)} \big( \div_{\mathbb G} b(z) - c(z) \big) u(z) \, dz \Bigr) \, d \r .
\end{align*}
In the statement of Theorem \ref{th-1} we took the exponent $\alpha=2$ because it is the usual one in the classical uniformly parabolic case. We recall that in Theorem 1.1 of \cite{CupiniLanconelli2021} other forms for the kernel $M_{\mathbb G}$ have been provided. 
}\end{remark} 

The mean value formulas in Theorem \ref{th-1} provide us with a simple proof of strong maximum and minimum principles for the operator $\L$ when $c \le 0$. We recall that an analogous result was obtained by using a barrier argument by Bony in \cite{Bony69} for H\"ormander's operators in the form \eqref{e0}, then by Amano in \cite{Amano79} for subelliptic operators with $C^1$ coefficients. In order to state the strong maximum and minimum principles, we introduce the notion of \emph{attainable set}. We say that a curve $ \g: [0,T] \rightarrow \R^{N+1}$ is \emph{$\L$-admissible} if it is absolutely continuous and 
\begin{equation*}
	 \dot{\g}(s) = \sum_{j=1}^m \omega_j (s) X_j (\gamma(s)) + X_{m+1} (\gamma(s))
\end{equation*}	
for almost every $s \in [0,T]$, with $\omega_1, \dots, \omega_m \in L^{2}([0,T])$.

\begin{definition} \label{def-prop-set}
Let $\O$ be any open subset of $\R^{N+1}$, and let $z_0 \in \O$. The \emph{attainable set} is 
\begin{equation*}
	\AS  ( \O ) = 
	\begin{Bmatrix}
	z \in \O \mid \hspace{1mm} \text{\rm there exists an 
	$\L$-admissible curve}
	\ \g : [0,T] \rightarrow \O \hspace{1mm} \\ 
	\hfill \text{\rm such that} \ \g(0) = z_0 \hspace{1mm} {\rm and}  
	\hspace{1mm} \g(T) = z
	\end{Bmatrix}.
\end{equation*}
We denote $\AS = \AS ( \O )$ whenever there is no ambiguity on the choice of the set $\O$. 
\end{definition}

We finally state a condition which relates the mean value formula to $\L$-admissible curves.

\begin{description}
\item[{\rm [H.4]}]
For every $z \in \R^{N+1}$ and $r>0$, and for every $\L$-admissible curve $\gamma$ such that $\gamma(0) = z$, there exists $s_0>0$ such that $\gamma(s) \in \Omega_r(z)$ for every $s \in ]0,s_0[$.
\end{description}

% Note that, in this case, the constant function $u(x,t) = 1$ is a solution to $\L u = 0$, so that the mean value formula  gives $\frac{1}{r^N}\int_{\Omega_r(z_0)}M(z_0;z)dz=1$.

\begin{theorem} \label{Thm-smp}
Let $\L$ be a differential operator satisfying all the hypotheses of Theorem \ref{th-1} and let $u$ be a classical solution to $\L u = f$ in an open subset $\O\subset\R^{N+1}$. Assume in addition that $c \le 0, c - \div_{\mathbb{G}}b  < 0$ and that {\rm [H.4]} holds. Let $z_0=(x_0,t_0) \in \O$ be such that $u (z_0) = \max_\Omega u \ge 0$ and $f \ge 0$ in $\Omega$; then 
\begin{equation*}
u(z) = u(z_0) \quad \text{and} \quad f(z) = u(z_0) c(z) \qquad \text{for every} \ z \in \overline {\AS  ( \O )}.
\end{equation*}
The analogous result holds true if $u (z_0) = \min_\Omega u \le 0$ and $f \le 0$ in $\Omega$. 
Moreover, we can drop the assumption on the sign of $u (z_0)$ if $c = 0$. 
\end{theorem}

Note that the condition [H.4] is satisfied by the examples considered in Section \ref{SecExamples}, where an application of Theorem \ref{Thm-smp} is given as well. The assumption $c - \div_{\mathbb{G}}b  < 0$ seems to be unnecessary for the validity of maximum and minimum principles. Indeed, in the article \cite{MalPalPol}, where a proof of the maximum principle for uniformly parabolic operators is based on the mean value formulas, this condition was removed by a suitable change of function. However we didn't succed to apply the same argument in the present setting as the structure of the operator $\L$ is very sensitive to analogous change of function.

\setcounter{equation}{0} 

\section{Functions of bounded variation}\label{secBV}

In this section we introduce some notation and the basic results on functions of bounded variation and sets with finite perimeter that we need to prove our mean value formulas. To simplify the notation, we put $n=N+1$ and denote by $\lambda_n$ the Lebesgue measure in $\R^n$. We also keep the notation $z=(x,t)$ for points in ${\mathbb R}^{N+1}$ and recall that the homogeneous dimension of ${\mathbb G}$ is denoted by ${\mathcal Q}$, see \eqref{e-delta}, \eqref{e-Q}. If $\mu$ is a Borel measure and $E$ is a Borel set, we use the notation $\mu \mres E(B) = \mu(E \cap B)$.

For an open set $\Omega\subset\R^n$ we define the space $BV_{\mathbb{G}}(\Omega)$ of functions of bounded variation in $\mathbb{G}$ following \cite{CapDanGar94The}. We refer to \cite{AmbFusPal00Fun} and to \cite{fraserser3} for more information on the Euclidean and the subriemannian case, respectively. 

\begin{definition}\label{defBV} 
Let $\Omega$ be an open subset of $\R^n$. For $u\in L^{1}({\Omega})$ we define
\begin{equation}  \label{deftotvarG}
\left\| \nabla_{\mathbb G} u \right\| \left(\Omega \right)
=\sup\left\{  \int_{{\mathbb R}^n}u(z) \, \mathrm{div}_{\mathbb G} g(z)dz:\ g\in C_{c}^{1}\left(\Omega,{\mathbb{R}}^{m+1}\right),\left\Vert g\right\Vert _{\infty}\leq1\right\}.
\end{equation}
We say that $u\in BV_{\mathbb{G}}(\Omega)$ if $\left\| \nabla_{\mathbb G} u \right\|(\Omega)$ is finite. 
\end{definition}

\begin{remark}\label{Xdependence}
{\rm We point out (see \cite[Remarks 2.10, 2.19]{fraserser3}) that the (usual) notation $\left\| \nabla_{\mathbb G} u \right\| \left(\Omega \right)$
is somehow misleading, as the total variation depends upon the fixed vector fields $X_{1}, \dots, X_{m+1}$, whereas the functional class $BV_{\mathbb{G}}(\Omega)$ only depends on $\mathbb{G}$ and $\Omega$. 
}\end{remark}

With the same proof contained e.g. in \cite[Prop. 3.6]{AmbFusPal00Fun}, it is possible to show that if $u$ belongs to $BV_\mathbb{G}(\Omega)$ then its total variation $\left\| \nabla_{\mathbb G} u \right\|$ is a finite positive Radon measure and there is a $\left\| \nabla_{\mathbb G} u \right\|$-measurable function $\sigma_{u}:\Omega \rightarrow {\mathbb{R}}^{m+1}$ such that $|\sigma_{u}(z)|_z=1$ for $\left\| \nabla_{\mathbb G} u \right\|$-a.e. $z\in \Omega$ and
\begin{equation}    \label{defsigmaf}
\int_{{\Omega}}u(z)\mathrm{div}_{\mathbb G}g(z)dz=
\int_{{\Omega}}\langle g,\sigma_{u}\rangle_z \,  d \left\| \nabla_{\mathbb G} u \right\|
\end{equation}
for all $g\in C_{\mathbb{G},c}^{1}(\Omega,{\mathbb{R}}^{m+1})$. We denote by $\nabla_{\mathbb G} u$ the vector measure $-\sigma_{u}\left\| \nabla_{\mathbb G} u \right\|$, so that $X_{j}u$ is the measure $(-\sigma_{u})_{j}\left\| \nabla_{\mathbb G} u \right\|$ and the following integration by parts formula holds true
\begin{equation}   \label{byparts_BV}
\int_{{\Omega}}u(z)X_{j}g(z)dz = - 
\int_{{\Omega}}g\left(z\right)d\left( X_{j}u\right) \left( z\right)
\end{equation}
for all $g\in C_{\mathbb{G},c}^{1}(\Omega)$.   

\begin{definition}[Sets of finite $\mathbb{G}$-perimeter] \label{perimeter}
Let $\chi_{E}$ be the characteristic function of the measurable set $E\subset\mathbb{R}^{n}$; we say that $E$ is a set of finite $\mathbb{G}$-perimeter in $\Omega$ if $\left\| \nabla_{\mathbb G} \chi_{E} \right\|(\Omega)$ is finite, and we call (generalized inward) $\mathbb{G}$-normal the $(m+1)$-vector 
\[
\nu_{E}(z) = - \sigma_{\chi_{E}}(z),
\]
which is defined $\|\nabla_{\mathbb G}\chi_E\|$-a.e.
\end{definition}
As customary, we write $P_\mathbb{G}(E,B)$ instead of $\left\| \nabla_{\mathbb G} \chi_{E} \right\|(B)$ for any Borel set $B$. Also, notice that if $A$ is open, then 
\begin{equation}\label{peronopens}
P_\mathbb{G}(E\cap A,A)=P_\mathbb{G}(E,A) ,
\end{equation}
see (2.25) in \cite{fraserser3}, and recall that 
$|\nu_{E}(z)|_z =1$ for $P_\mathbb{G}(E)$-a.e. $z\in\mathbb{R}^{n}.$ With this notation, (\ref{defsigmaf}) takes the form 
\begin{equation}   \label{div thm}
\int_{E}\mathrm{div}_{\mathbb G}g(z)dz=-
\int_{{\Omega}}\langle g,\nu_{E}\rangle_z dP_{\mathbb G}(E),
\end{equation}
for all $g\in C_{\mathbb{G},c}^{1}(\Omega,\R^{m+1})$. 

We refer to \cite[Theorem 2.3.5]{fraserser} for a proof of the following statement that connects the total variation of a $BV_{\mathbb G}$ function with the perimeter of its level sets. 

\begin{proposition} [Coarea formula]\label{coarea} 
If $u\in BV_{\mathbb{G}}(\Omega)$ for some open set $\Omega \subset \R^n$ then for a.e.
$\tau\in\mathbb{R}$ the set $E_{\tau}=\{x\in{\Omega}:\ u(x)>\tau\}$ has finite $\mathbb{G}$-perimeter in $\Omega$ and
\begin{equation}    \label{coareaformula}
\left\| \nabla_{\mathbb G} u \right\| (\Omega)=
\int_{-\infty}^{+\infty} P_{\mathbb G}(E_\tau,\Omega)d\tau.
\end{equation}
Conversely, if $u\in L^{1}(\Omega)$ and 
$\int_{-\infty}^{+\infty}P_{\mathbb G}(E_\tau,\Omega)d\tau<\infty$, then $u\in BV_\mathbb{G}(\Omega)$ and equality \eqref{coareaformula} holds. Moreover, if $g:\Omega \rightarrow\mathbb{R}$ is a Borel function, then
\begin{equation}    \label{coareag}
\int_{\Omega}g(z)d \left\| \nabla_{\mathbb G} u \right\| (z)=
\int_{-\infty}^{+\infty}\int_{\Omega}g(z)d P_{\mathbb G}(E_\tau)(z)d\tau. 
\end{equation}
\end{proposition}
Let us come to some finer properties of $BV_\mathbb{G}$ functions and perimeters. In order to put formula \eqref{div thm} in a form closer to the classical one we define the {\em measuretheoretic} or {\em essential} boundary. 
\begin{definition}[Essential boundary]\label{DefEssBdry}
Let $E\subset \R^n$ be a measurable set. We say that $z\in \partial_{\mathbb G}^*E$ if 
\[
\limsup_{r\to 0}\frac{\lambda_n(B_r(z) \cap E)}{\lambda_n(B_r(z))}>0,\qquad
\limsup_{r\to 0}\frac{\lambda_n(B_r(z)\setminus E)}{\lambda_n(B_r(z))}>0 
\]
and we call $\partial_{\mathbb G}^*E$ the {\em measuretheoretic} or {\em essential} boundary of $E$.
\end{definition}

It is immediately checked that $\partial_{\mathbb G}^*E \subset \partial E$. Observe that two different but equivalent distances on ${\mathbb G}$ give the same essential boundary.  

Let us see that the divergence theorem \eqref{div thm} can be rewritten in a form much closer to the classical formula, see \cite[Theorems 5.3, 5.4]{ambr2}, where the problem is settled in general metric measure spaces, and \cite[Theorem 4.16]{AmbKleLeD09JGA}.

\begin{theorem}  \label{hausdorffrep}
Given a set of finite ${\mathbb G}$-perimeter $E\subset\R^n$, for $P_{\mathbb G}(E,\cdot)$-a.e. $z\in\R^n$ there is $\bar{r}(z)>0$ such  that
\[
\ell_{\mathbb G} r^{{\mathcal Q}-1} \leq P_{\mathbb G}(E, B(z,r)) \leq L_{\mathbb G}r^{{\mathcal Q}-1}
\]
for every $r<\bar{r}(z)$, where $0<\ell_{\mathbb G} \leq L_{\mathbb G}<\infty$ are two constants depending only on the group. As a consequence, $P_{\mathbb G}(E,\cdot)$ is concentrated on $\partial_{\mathbb G}^*E$, i.e., $P_{\mathbb G}(E,{\mathbb G}\setminus\partial_{\mathbb G}^*E)=0$, and  
there is a Borel function
$\beta_E:{\mathbb R}^n\to [\ell_{\mathbb G},L_{\mathbb G}]$ such that
\begin{equation}\label{intrepr}
P_{\mathbb G}(E,B)=\int_{B\cap \partial_{\mathbb G}^*E}\beta_E (z)\,d{\mathcal S}_{\mathbb G}^{{\mathcal Q}-1}(z),\quad \forall B \text{\ \ Borel set}.
\end{equation}
\end{theorem}
The above theorem allows us to rewrite formula \eqref{div thm} as an integral on the essential boundary with respect to the $({\mathcal Q}-1)$-dimensional spherical Hausdorff measure as follows:
\begin{equation}\label{div thm boundary}
\int_{E}\mathrm{div}_{\mathbb G}g(z)dz=
-\int_{\partial_{\mathbb{G}}^{\ast}E}\langle
g,\nu_{E}\rangle_z \, \beta_E(z)\, d{\mathcal S}_{\mathbb G}^{{\mathcal Q}-1} .
\end{equation}
In the following remarks we collect some useful results proved by Franchi, Serapioni and Serra Cassano \cite[Theorem 2.3.5]{fraserser} and V. Magnani \cite{MagnaniIUMJ} on functions belonging to $C^1_{\mathbb G}(\Omega)$, for which much more information is available. 

\begin{remark}\label{smoothcase}{\rm 
If $\Omega$ is bounded, a function $u$ in $C^1_{\mathbb G}(\Omega)$ also belongs to $BV_{{\mathbb G},{\rm loc}}(\Omega)$ and by \eqref{byparts_BV} the equalities
\[
\int_\Omega X_j^*g(z)u(z)dz = \int_\Omega g(z)X_ju(z)dz, \qquad j=1,\ldots,m+1,
\]
hold for every $g\in C^1_c(\Omega)$. Recalling \eqref{e-adj}, we find that the measure derivative of $u$ is $\nabla_{\mathbb G}u\,\lambda_n$. Moreover, we say that $S\subset\Omega$ is a ${\mathbb G}$-regular surface if for any $p\in S$ there are an open neighborhood $U$ of $p$ and $f\in C^1_{\mathbb G}(U)$ such that
\[
S\cap U = \{z\in U:\ f(z)=0\ \textrm{and}\ \nabla_{\mathbb G}f(z)\neq 0\}.
\]
Let $\Omega$ be an open subset of ${\mathbb R}^n$, $f\in C^1_{\mathbb G}(\Omega)$, $E=\{f<0\}$, $S=\{f=0\}$, and let $p\in\Omega$ be such that $f(p)=0$ and $\nabla_{\mathbb G}f(p)\neq 0$. Then, as proved in \cite[Theorem 2.1]{fraserser2}, there is a neighborhood $U$ of $p$ such that $S\cap U$ has finite perimeter and
\begin{equation}\label{nuGamma}
\nu_E(z)=-\frac{\nabla_{\mathbb G}f(z)}{|\nabla_{\mathbb G}f(z)|_z}, \qquad  z\in S\cap U.
\end{equation}
In such a situation the equality $\partial^*_{\mathbb G}(E\cap U) = \partial(E\cap U)$ holds, see \cite[Theorem 3.3]{fraserser2}. Notice also that the topological dimension of a $C^1_{\mathbb G}$-regular surface is $n-1$, see \cite[Proposition 3.1]{fraserser2}, whereas its Hausdorff dimension with respect to the distance $d_\infty$ is ${\mathcal Q}-1$, see \cite[Corollary 3.7]{fraserser2}.
}\end{remark}

\begin{remark}\label{thetacostante}
{\rm If $E$ is a finite perimeter set and $\partial_{\mathbb G}^*E$ is ${\mathbb G}$-regular, then formulas \eqref{intrepr} and \eqref{div thm boundary} become simpler. Indeed, in this case the normal unit vector $\nu(z)$ is defined for every $z\in \partial_{\mathbb G}^*E$ and the function $\beta_E$ is constant, $\beta_E(z)=1$ for every $z\in\partial^*E$, by Theorem 4.1 in \cite{MagnaniIUMJ} and the definition of the constant $\theta$ in \eqref{deftheta}. This is the reason why we have chosen the distance $d_\infty$ and we have normalized the Hausdorff measure. These considerations are important in our proof of Theorem \ref{th-1}, where \eqref{div thm boundary} is applied to sets with finite perimeter such that \emph{a part of} the essential boundary is ${\mathbb G}$-regular. Indeed, Theorem 4.1 in \cite{MagnaniIUMJ} is local, hence if $F\subset\partial_{\mathbb G}^*E$ is ${\mathbb G}$-regular and relatively open, then $\beta_E=1$ in $F$.
}\end{remark}

We end this section with a variant of the {\em localization lemma}, see \cite[Proposition 3.56]{AmbFusPal00Fun} for the Euclidean case and \cite[Lemma 2.21]{fraserser3} for the case of groups, where balls instead of hyperplanes are considered. 

\begin{lemma}\label{localization}
Let $E\subset \R^n$ be a set of finite ${\mathbb G}$-perimeter. Then, for $\lambda_1$-a.e. $\tau\in{\mathbb R}$ the set $E\cap\{t<\tau\}$ has finite ${\mathbb G}$-perimeter and for every Borel set $B$ we have 
\[  
\nabla_{\mathbb G}\chi_{E\cap\{t<\tau\}}(B)=
\nabla_{\mathbb G}\chi_E(B\cap\{t<\tau\})
-\lambda_N(B\cap\{t=\tau\}) e_{N+1}.
\]
\end{lemma}
\begin{proof} Setting $u_s(x,t)=\{[(\tau -t)/s)\vee 0]\wedge 1\}\chi_E(x,t)$, we have $u_s\to \chi_{E\cap\{t<\tau\}}$ in $L^1({\mathbb R}^n)$ as $s\downarrow 0$, whence by the semicontinuity of the total variation,
\begin{equation*}
P_{\mathbb G}(E\cap\{t<\tau\})\leq \liminf_{s\downarrow 0} \| \nabla_{\mathbb G} u_s \|({\mathbb R}^n).
\end{equation*}
Now we compute 
\[
 \nabla_{\mathbb G} u_s = 
 \{[(\tau -t)/s)\vee 0]\wedge 1\}
 \nabla_{\mathbb G} \chi_E\mres\{t <\tau-s\} - \frac{1}{s}
\lambda_n\mres (E\cap\{\tau-s< t <\tau\})e_{N+1}  
\]
hence, setting $m(\tau)=\lambda_n(E\cap\{t<\tau\})$, we have 
\[
\limsup_{s\downarrow 0} \| \nabla_{\mathbb G} u_s \|({\mathbb R}^n) \leq \|\nabla_{\mathbb G} \chi_E\|(\{t <\tau\}) + m'(\tau),
\]
where $m'$ denotes the right derivative of $m$, which is finite for $\lambda_1$-a.e. $\tau\in\R$. This proves that $E\cap\{t<\tau\}$ has finite ${\mathbb G}$-perimeter for $\lambda_1$-a.e. $\tau\in\R$. Therefore, since by \eqref{peronopens}
\[
P_{\mathbb G}(E\cap\{t<\tau\},\{t<\tau\})
= P_{\mathbb G}(E,\{t<\tau\}),
\] 
and ${\mathcal S}_{\mathbb G}^{{\mathcal Q}-1}(\partial_{\mathbb G}^*E\cap\{t=\tau\})=0$ for a.e. $\tau$, we have 
\[
\nabla_{\mathbb G}\chi_{E\cap\{t<\tau\}}(B)=
\nabla_{\mathbb G}\chi_E(B\cap\{t<\tau\})
+\nabla_{\mathbb G}\chi_E(B\cap\{t=\tau\})
\] 
for any Borel set $B$ and $\lambda_1$-a.e. $\tau$. But, $\nabla_{\mathbb G}\chi_E(z)=-e_{N+1}{\mathcal S}_{\mathbb G}^{{\mathcal Q}-1} = -e_{N+1}{\mathcal L}^N$ for $z\in E\cap\{t=\tau\}$ and the thesis follows.
\end{proof}

\setcounter{equation}{0} 

\section{Proof of the mean value formula}\label{secProofs}

In this section we give the proof of the mean value formulas and of the strong maximun principle. In what follows, $\Omega$ is an open subset of $\R^{N+1}$, $z_0=(x_0,t_0) \in \Omega$, $\Gamma^*(z; z_0)$ is the fundamental solution of $\L^*$, and $r_0>0$ is such that $\overline{\Omega_{r_0}(z_0)}$ is a bounded subset of $\Omega$. 

First, we state the following consequence of Lemma \ref{localization}. It applies to the set $\Omega_r(x_0,t_0) \cap \big\{ t < t_0 - \varepsilon \big\}$, see Fig. 1 below.

\begin{tikzpicture}
\clip (-.5,8) rectangle (6.7,2);
\shadedraw [top color=black!10] (-2,6) rectangle (7,1); % node[below] {  $t = t_0 - \varepsilon$}; 
\begin{axis}[axis y line=middle, axis x line=middle, 
    xtick=\empty,ytick=\empty, % xlabel=$x$, %ylabel=$t$, 
    ymin=-1.1, ymax=1.1,
    xmin=-.2,xmax=1.8, samples=101, rotate= -90]
\addplot [black,line width=.7pt, domain=-.01:.01] {sqrt(.0001 - x * x)} node[above] {\hskip12mm $(x_0,t_0)$};
\addplot [black,line width=.7pt, domain=-.01:.01] {-sqrt(.0001 - x * x)};
\addplot [black,line width=.7pt, domain=.001:1] {sqrt(- 2 * x * ln(x))}; 
\addplot [black,line width=.7pt,domain=.001:1] {- sqrt(- 2 * x * ln(x))} 
node[below] { \hskip20mm $\Omega_r(x_0,t_0)$};
%node[pos=0.87,below,sloped]{$y=\sinh\alpha$};
\end{axis}
\draw [<-,line width=.4pt] (2.8475,7) -- (2.8475,2);
\draw [line width=.7pt,] (-1,6) -- (6.7,6) node[below] { \hskip-18mm  $t = t_0 - \varepsilon$}; 
\end{tikzpicture} 

{\sc Fig.1}  - The set $\Omega_r(x_0,t_0) \cap \big\{ t < t_0 - \varepsilon \big\}$.

\bigskip

\begin{proposition}\label{PropLocalization}
With the notation above, for a.e. $\varepsilon >0$ and $r \in ]0, r_0]$ we have 
\begin{align*}
\int_{\Omega_r(z_0)\cap\{ t< t_0 - \varepsilon \}} {\rm div}_{\mathbb G} \, \Phi \, dz= &
\int_{\partial_{\mathbb G}^*\Omega_r(z_0)\cap\{ t< t_0 - \varepsilon \}} 
\langle \Phi,\nu_{\Omega_r(z_0)}\rangle_z \beta_{\Omega_r(z_0)} d{\mathcal S}_{\mathbb G}^{{\mathcal Q}-1} \\
%\theta_{\Omega_r(z_0)\cap\{t< t_0 - \varepsilon \}}d{\mathcal H}^{Q+1} \\
& \qquad \qquad +\int_{\Omega_r(z_0)\cap\{ t = t_0 - \varepsilon \}}
\langle \Phi,e_{N+1}\rangle_z d{\lambda}_{N} 
\end{align*}
for every $\Phi \in C^1_{\mathbb G}(\Omega)$.
\end{proposition}
\begin{proof} 
As $\Gamma^* (\cdot; z_0)\in BV_{{\mathbb G},\rm loc}(\Omega)$, see Remark \ref{smoothcase}, by the coarea formula \eqref{coareaformula} for a.e. $r \in ]0, r_0]$ the set $\Omega_r(z_0)$ has finite ${\mathbb G}$-perimeter. We then apply Lemma \ref{localization} with $E=\Omega_r(z_0)$ and $\tau = t_0 - \varepsilon$ to have that ${\Omega_r(z_0)\cap\{ t< t_0 - \varepsilon \}}$ has finite perimeter. The conclusion follows from \eqref{div thm boundary}. 
\end{proof}

\medskip

\begin{proof}{\sc of Theorem \ref{th-1}.} Let $u$ be a classical solution to $\L u = f$ in $\Omega$ and let $\Omega_r(z_0)$ be such that $\overline{\Omega_r(z_0)}$ is a compact subset of $\Omega$. We assume, as it is not restrictive, that $u$ vanishes out of a compact subset of $\Omega$ so that $u$ can be smoothly extended by setting $u(z) = 0$ for every $z \not \in \Omega$. We prove our claim by applying Proposition \ref{PropLocalization} with $r \in ]0,r_0]$ and $t=t_0-\varepsilon_k$, for some monotone sequence $(\varepsilon_k)_{k \in \mathbb{N}}$ such that $\varepsilon_k \to 0$ as $k \to \infty$. Of course, we choose $r$ and $\varepsilon_k$ such that the statement of Proposition \ref{PropLocalization} holds true.

For this choice of $r$, we set $v(z) := \Gamma^*(z;z_0) - \frac{1}{r}$, and we note that 
\begin{equation} \label{eq-div-L*}
\begin{split}
 u(z) \L^* v(z) - v(z) \L u(z) =& 
 \sum_{i,j=1}^m X_j \big( u(z) a_{ij}(z) X_i v(z) - v(z) a_{ij}(z) X_i u(z) \big) 
\\
& -\sum_{j=1}^m X_j(u(z)v(z)b_j(z)) - X_{m+1}(u(z) v(z))
\end{split}
\end{equation}
for every $z \in \Omega \backslash \big\{z_0 \big\}.$ We then recall that $\L^* v = \frac{1}{r} \left( \div_{\mathbb G} \, b - c \right)$ and $\L u  = f$ in $\Omega \backslash \big\{z_0 \big\}$. 
Then \eqref{eq-div-L*} can be written as follows
\begin{equation*}
 \frac{1}{r} \left( \div_{\mathbb G} b - c \right) u - v f = 
 \div_{\mathbb G} \Phi , \qquad 
 \Phi:=(u A\nabla_{\mathbb G} v - v A\nabla_{\mathbb G} u - uv b, - uv).
\end{equation*}
We then apply Proposition \ref{PropLocalization} and we find
\begin{equation} \label{eq-div-k}
\begin{split}
 &\int_{\Omega_r(z_0) \cap \{ t<t_0-\varepsilon_k \}}
\Bigl(\tfrac{1}{r} ( \div_{\mathbb G} b(z) - c(z) ) u(z) - v(z) f(z)\Bigr) dz
\\ 
& \qquad = - \int_{\partial_{\mathbb G}^*\Omega_r(z_0) \cap \{ t<t_0-\varepsilon_k \}}
\!\!\!\!\!\!\!\!\!\! \scp{\Phi, \nu_{\Omega_r(z_0)}}_z \beta_{\Omega_r(z_0)}
d{\mathcal S}_{\mathbb G}^{{\mathcal Q}-1} 
+ \int_{\Omega_r(z_0) \cap \{ t=t_0-\varepsilon_k \}} 
\!\!\!\!\!\!\!\!\!\! \scp{\Phi, e_{N+1}}_z d{\lambda}_{N}.
\end{split}
\end{equation} 
We next let $k \to \infty$ in the above identity. As $v \in L^1({\Omega_r(z_0)})$, and the remaining functions appearing in the left hand side of \eqref{eq-div-k} are bounded and continuous on $\overline{\Omega_r(z_0)}$, we plainly have 
\begin{equation} \label{eq-div-1}
\begin{split}
\lim_{k \to \infty} 
\int_{\Omega_r(z_0)\cap\left\{ t<t_0-\varepsilon_k\right\}}
 &\left( \tfrac{1}{r} \left( \div_{\mathbb G} b(z) - c(z) \right) u(z) - v(z) f(z)\right) dz = \\
& \int_{\Omega_r(z_0)} \left( \tfrac{1}{r} \left( \div_{\mathbb G} b(z) - c(z) \right) u(z) - v(z) f(z)\right) dz.
\end{split}
\end{equation}
We next prove that 
\begin{equation} \label{eq-div-2}
 \lim_{k \to + \infty} \int_{\Omega_r(z_0) \cap \left\{ t=t_0-\varepsilon_k \right\}} \scp{\Phi, e_{N+1}}_z d\lambda_{N} = u(z_0).
\end{equation}
We have $\scp{\Phi, e_{N+1}}_z (z) = u(z) v(z)$, then 
\begin{equation} \label{eq-psir}
 \int_{\Omega_r(z_0) \cap \left\{ t=t_0-\varepsilon_k \right\}} \!\!\scp{\Phi, e_{N+1}}_z d \lambda_{N} =
 \int_{\widetilde{\mathcal{I}}_k} \!\! \!\! 
 u(x,t_0- \varepsilon_k) \left( \Gamma^*(x,t_0 - \varepsilon_k;x_0, t_0) - \frac{1}{r} \right) d x,
\end{equation}
where we have denoted
\begin{equation*} %\label{eq-div-2}
 \widetilde{\mathcal{I}}_k := \left\{ x \in \R^N \mid (x,t_0-\varepsilon_k) \in \overline{\Omega_r (z_0)} \right\}.
\end{equation*}
Note that $\widetilde{\mathcal{I}}_k$ agrees with the set $\mathcal{I}_{r, \varepsilon} (z_0)$ defined in \eqref{eq-Irepsilon}, with $\varepsilon = \varepsilon_k$, then the following assertion holds because of our assumption \eqref{eq-claim-2}
\begin{equation} \label{eq-claim-2bis}
\lim_{k \to \infty} \H^{N}_e \left( \widetilde{\mathcal{I}}_k \right) = 0, \qquad
 \lim_{k \to \infty}  \int_{\R^{N} \backslash \widetilde{\mathcal{I}}_k} \!\! \!\! 
 \Gamma^*(x,t_0 - \varepsilon_k;x_0, t_0) d x = 0.
\end{equation}
Since $\Gamma^*$ is the fundamental solution to $\L^* v = 0$, and $u$ is bounded and continuous, we have
\begin{equation*} %\label{eq-div-2}
 \lim_{k \to \infty}  \int_{\R^N} \!\! \!\! \Gamma^*(x,t_0 - \varepsilon_k;x_0, t_0) u(x,t_0)  d x = u(x_0,t_0).
\end{equation*}
The conclusion of the proof of \eqref{eq-div-2} then follows from \eqref{eq-claim-2bis}, by using again the fact that $u$ is bounded and uniformly continuous.

We are left with the first integral in the right hand side of \eqref{eq-div-k}. We preliminarily note that $v(z) = 0$ for every $z \in \partial \Omega_r(z_0)$, then
\begin{equation*}
 \Phi(z) = \big( u (z) A(z) \nabla_{\mathbb G} v(z), 0 \big) \quad \forall  z \in \partial \Omega_r(z_0).
\end{equation*}
Moreover, \eqref{nuGamma} gives $\nu(z)= - \tfrac{\nabla_{\mathbb G} \Gamma^*(z;z_0) }{|\nabla_{\mathbb G} \Gamma^*(z;z_0)|_z}$ for every $z$ such that $\nabla_{\mathbb G}\Gamma^*(z;z_0)\neq 0$, while $\Phi(z) = 0$ whenever $\nabla_{\mathbb G}\Gamma^*(z;z_0)= 0$.
We then find
\begin{equation} \label{eq-div-crit}
 \int_{\partial_{\mathbb G}^*\Omega_r(z_0) \cap \{ t<t_0-\varepsilon_k \}} 
 \!\!\!\!\!\!\!\!\!\!\!\!\!\!\!\!\!\!\! \scp{\Phi,\nu}_z 
 \beta_{\Omega_r(z_0)} 
 d{\mathcal S}_{\mathbb G}^{{\mathcal Q}-1}
= \int_{\partial_{\mathbb G}^*\Omega_r(z_0) \cap \left\{ t<t_0-\varepsilon_k \right\}} 
\!\!\!\!\!\!\!\!\!\!\!\!\!\!\!\!
u(x,t) K_{\mathbb G} (z_0;z) 
\beta_{\Omega_r(z_0)}
d{\mathcal S}_{\mathbb G}^{{\mathcal Q}-1},
\end{equation}
where 
\begin{equation*}
	K_{\mathbb G} (z_0; z) = 
	\frac{\langle A(z) \nabla_{\mathbb G} \Gamma^*(z;z_0), \nabla_{\mathbb G} \Gamma^*(z;z_0) \rangle_z }
	{|\nabla_{{\mathbb G}}\Gamma^*(z;z_0)|_z}
\end{equation*}
is the kernel defined in \eqref{e-kernels}. We next prove that $K_{\mathbb G} (z_0; \cdot)$ belongs to $L^1$ with respect to the measure ${\mathcal S}_{\mathbb G}^{{\mathcal Q}-1}\mres \partial^*_{\mathbb G}\Omega_r(z_0)$. We apply the identity \eqref{eq-div-k} to a compactly supported smooth function $u$ with the property that $u(z) = 1$ for every $z \in \overline {\Omega_r(z_0)}$. We have $\L u = c$ in $\overline {\Omega_r(z_0)}$, then \eqref{eq-div-1} and \eqref{eq-div-2} yield
\begin{equation*} %\label{eq-div-3}
\lim_{k \to + \infty} \int_{\partial_{\mathbb G}^*\Omega_r(z_0) \cap \left\{ t<t_0-\varepsilon_k \right\}} 
\!\!\!\!\!\!\!\!\!\!\!\!\!\!\!\!\!\!\!\!\! K_{\mathbb G} (z_0;z)  \beta_{\Omega_r(z_0)} d{\mathcal S}_{\mathbb G}^{{\mathcal Q}-1} = 1 + \int_{\Omega_r(z_0)} 
\Big(c(z) \Gamma^*(z;z_0) - \frac{1}{r} \div_{\mathbb{G}} \, b(z) \Big) d z.
\end{equation*}
Since the functions $K_{\mathbb G}$ and $\beta_{\Omega_r(z_0)}$ are both non-negative and the sequence $\big(\varepsilon_k\big)_{k \in \mathbb{N}}$ is decreasing, we conclude that 
\begin{equation*} %\label{eq-div-3}
\int_{\partial_{\mathbb G}^*\Omega_r(z_0)} K_{\mathbb G} (z_0;z)  \beta_{\Omega_r(z_0)} d{\mathcal S}_{\mathbb G}^{{\mathcal Q}-1}
\end{equation*}
is finite. This proves the first equality in the following  
\begin{equation} \label{eq-div-3}
\begin{split}
 \lim_{k \to + \infty} \int_{\partial_{\mathbb G}^*\Omega_r(z_0) \cap \{ t<t_0-\varepsilon_k \}} 
 \!\!\!\!\!\!\!\!\!\!\!\!\!\!\!\!\!\!\! \scp{\Phi,\nu}_z \beta_{\Omega_r(z_0)} 
d{\mathcal S}_{\mathbb G}^{{\mathcal Q}-1} 
& = \int_{\psi_r(z_0)} 
\!\!\!\!\!\!\! u(x,t) K_{\mathbb G} (z_0;z)  \beta_{\Omega_r(z_0)} 
d{\mathcal S}_{\mathbb G}^{{\mathcal Q}-1}
\\
%&= \int_{\psi_r(z_0)\setminus\{\nabla_{\mathbb G}%\Gamma^*=0\}} 
%\!\!\!\!\!\!\! u(x,t) K_{\mathbb G} (z_0;z) \beta_{\Omega_r(z_0)} 
%d{\mathcal S}_{\mathbb G}^{{\mathcal Q}-1}
%\\
& = \int_{\psi_r(z_0)} u(x,t) K_{\mathbb G} (z_0;z)  
d{\mathcal S}_{\mathbb G}^{{\mathcal Q}-1}
\end{split}
\end{equation}
for every $u \in C(\overline{\Omega_r(z_0)})$. In the second equality we took into account that $K_{\mathbb G} (z_0;z)=0$ if $\nabla_{\mathbb G}\Gamma^*(z)=0$ and that $\psi_r(z_0)\setminus\{\nabla_{\mathbb G}\Gamma^*=0\}$ is a $C^1_{\mathbb G}$-regular surface, hence $\beta_{\Omega_r(z_0)}=1$ there, see Remark \ref{thetacostante}. The proof of the first assertion of Theorem \ref{th-1} then follows by using \eqref{eq-div-1}, \eqref{eq-div-2} and \eqref{eq-div-3} in \eqref{eq-div-k}. 

The proof of the second assertion of Theorem \ref{th-1} is a direct consequence of the first one and of the coarea formula stated in Proposition \ref{coarea}. Indeed, fix a positive $r$ as above, multiply by $\frac{1}{r}$ 
and integrate over $]0,r[$. We find
\begin{equation} \label{e-meanvalue-step1}
\begin{split}
	\frac{1}{r} \int_0^r u(z_0) d \varrho = & 
	\frac{1}{r} \int_0^r  
	\bigg(\int_{\partial_{\mathbb G}^* \Omega_\varrho(z_0)} 
	K_{\mathbb G} (z_0;z) u(z) \, d P_{\mathbb G}(\Omega_\r(z_0))\bigg) d \varrho\, 
	\\
	& + \frac{1}{r} \int_0^r \bigg(\int_{\Omega_\r(z_0)} f (z) 
	\left( \tfrac{1}{\r} - \Gamma^*(z;z_0) \right) dz \bigg) d \varrho  
	\\
	& + \frac{1}{r} \int_0^r \tfrac{1}{\r} \bigg( 
	\int_{\Omega_\r (z_0)} \left( \div_{\mathbb G} b(z) - c(z) \right) u(z) \, dz \bigg) d \varrho.
\end{split}
\end{equation}
The left hand side of the above equality equals $u(z_0)$, while the last two terms agree with the last two terms appearing in the statement of Theorem \ref{th-1}. In order to conclude the proof we only need to show that 
\begin{equation} \label{e-meanvalue-step2}
	  \int_0^r \bigg(\int_{\partial_{\mathbb G}^* \Omega_\varrho(z_0)} 
	  K_{\mathbb G} (z_0;z) u(z) \, d P_{\mathbb G}(\Omega_\r(z_0)) \bigg) d \varrho 
	= \int_{\Omega_\r (z_0)} M_{\mathbb G} (z_0;z) u(z) dz
\end{equation}
where $M_{\mathbb G}$ is the kernel defined in \eqref{e-kernels}.
With this aim, we set 
\begin{equation*}
 E_y(z_0) := \big\{ z \in \R^{N+1} \mid \Gamma^*(z;z_0) > y \big\}, \quad y>0, 
\end{equation*}
and we substitute $y = \frac{1}{\varrho}$ in the left hand side of \eqref{e-meanvalue-step2}. Note that $\partial_{\mathbb G}^* E_y(z_0) = \partial_{\mathbb G}^* \Omega_\r(z_0)$ if $y = \frac{1}{\varrho}$ and $\Gamma^*(z;z_0)= y$ for every $z \in \partial_{\mathbb G}^* E_y(z_0)$. Then 
\begin{equation*} %\label{e-meanvalue-step3}
\begin{split}
	  \int_0^r &\bigg( \int_{\partial_{\mathbb G}^* \Omega_\r(z_0)} \!\!\!\!\!
 	  \frac{\langle A(z) \nabla_{\mathbb G} \Gamma^*(z;z_0), \nabla_{\mathbb G} \Gamma^*(z;z_0) \rangle_z }
 	{|\nabla_{\mathbb G} \Gamma^*(z_0;z)|_z} u(z) \, d P_{\mathbb G}(\Omega_\r(z_0)) \bigg) d \varrho 
 	\\
      & = 
     \int_{\frac{1}{r}}^{\infty} \frac{1}{y^{2}} \bigg(\int_{\partial_{\mathbb G}^* E_y(z_0)} 
     \!\!\!\!\!\!\!
  	\frac{\langle A(z) \nabla_{\mathbb G} \Gamma^*(z;z_0), \nabla_{\mathbb G} \Gamma^*(z;z_0) \rangle_z }
  	{|\nabla_{\mathbb G}\Gamma^*(z;z_0)|_z}u(z) \, d P_{\mathbb G}(E_y(z_0)) \bigg)  dy
  	\\
	& =\int_{\frac{1}{r}}^{\infty} \bigg(\int_{\partial_{\mathbb G}^* {E_y(z_0)}}
	\frac{\langle A(z) \nabla_{\mathbb G} \Gamma^*(z;z_0), \nabla_{\mathbb G} \Gamma^*(z;z_0) \rangle_z }
	{\Gamma^*(z;z_0)^2 {|\nabla_{\mathbb G}\Gamma^*(z;z_0)|_z}} u(z) \, d P_{\mathbb G}(E_y(z_0)) \bigg) d y.
\end{split}
\end{equation*}
We finally recall the definition of the kernel $M_{\mathbb G}$ and we conclude the proof of \eqref{e-meanvalue-step2} by using the coarea formula stated in Proposition \ref{coarea}.
\end{proof}

\medskip

\begin{proof} {\sc of Theorem \ref{Thm-smp}.}  
We first note that $\L \, 1 = c$, then Theorem \ref{th-1} yields
\begin{equation*} %\label{e-meanvalue}
\begin{split}
  	\frac{1}{\r} \int_{\Omega_\r(z_1)} M_{\mathbb G} (z_1; z) \, dz + \frac{1}{\r} \int_0^{\r} 
	\Bigl(& \tfrac{1}{s} \int_{\Omega_s (z_1)}  \left( \div_{\mathbb G} \, b(z) - c(z) \right) \, dz \Big) d s  \\
	 &+ \int_{\Omega_s (z_1)} c (z) \left( \tfrac{1}{s} - \Gamma^*(z;z_1) \right) dz \Bigr)\, d s = 1
\end{split}
\end{equation*}
for every $z_1 \in \Omega$ and $\r >0$ such that $\overline{\Omega_\varrho(z_1)} \subset \Omega$.

We claim that for every $z_1\in\O$ such that $u(z_1) = \max_\Omega u$ we have 
\begin{equation} \label{eq-claim-smp}
 u(z) = u(z_1) \qquad \text{for every} \quad z \in \overline{\Omega_\varrho(z_1)}.
\end{equation}
By using again Theorem \ref{th-1} and the above identity we obtain
 \begin{align*} %\label{e-meanvalue}
	0 = & \frac{1}{\varrho} \int_{\Omega_\varrho(z_1)} M_{\mathbb G} (z_1; z) \big(u(z)- u(z_1)\big) \, dz 
	\\
	& + \frac{1}{\r} \int_0^{\r} \tfrac{1}{s}  \Big( \int_{\Omega_s (z_1)} 
	\left( \div_{\mathbb G} \, b(z) - c(z) \right) \big(u(z)- u(z_1)\big) \, dz \Big) d s 
	\\
	& + \frac{1}{\r} \int_0^{\r} \left(
	\int_{\Omega_s (z_1)} (f (z) - u(z_1) c(z)) \left( \tfrac{1}{s} - \Gamma^*(z;z_1) \right) dz \right) d s.
\end{align*}
Note that $c \le 0, f \ge 0, \div_{\mathbb G} \, b(z) - c(z)> 0$ and $u(z) \le u(z_1)$, being $u(z_1) = \max_{\Omega} u \ge 0$. Moreover, $M_{\mathbb G}(z_1;z) \ge 0$ and $\Gamma^*(z;z_1) > \tfrac{1}{s}$ for every $z \in \Omega_s(z_1)$, then
 \begin{align*} %\label{e-meanvalue}
	0 \ge & \frac{1}{\varrho} \int_{\Omega_\varrho(z_1)} M_{\mathbb G} (z_1; z) \big(u(z)- u(z_1)\big) \, dz 
	\\
	0 \ge &  \frac{1}{\r} \int_0^{\r} \tfrac{1}{s} \Big( \int_{\Omega_s (z_1)} 
	\left( \div_{\mathbb G} \, b(z) - c(z) \right) \big(u(z)- u(z_1)\big) \, dz \Big) d s 
	\\
	0 \ge & \frac{1}{\r} \int_0^{\r} \left(
	\int_{\Omega_s (z_1)} (f (z) - u(z_1) c(z)) \left( \tfrac{1}{s} - \Gamma^*(z;z_1) \right) dz \right) d s.
\end{align*}
%every integral in the above identity is non-positive. 
Hence the three integral vanish and, as a consequence, 
%$M (z_1; z)_{\mathbb G} \big((u(z)- u(z_1)\big)=0$
$\left( \div_{\mathbb G} \, b(z) - c(z) \right) \big((u(z)- u(z_1)\big)=0$ for $\lambda_{N+1}$ almost every $z \in \Omega_\r(z_1)$. Because our assumption on the sign of $\div_{\mathbb{G}} \, b - c$ we have that 
%$M_{\mathbb G} (z_1; z) \ne 0$ for $\lambda^{N+1}$ almost every $z \in \Omega_\r(z_1)$. As a consequence 
$u(z)= u(z_1)$ for $\lambda^{N+1}$ almost every $z \in \Omega_\r(z_1)$, and \eqref{eq-claim-smp} follows from the continuity of $u$.

We are in position to conclude the proof of Theorem \ref{Thm-smp}. Let $z$ be a point of $\AS  ( \O )$, and let $\gamma: [0,T] \to \O$ be an $\L$-admissible path such that $\g(0)= z_0$ and $\g(T) = z$. We prove that $u(\gamma(t)) = u(z_0)$ for every $t \in [0,T]$. Let
\begin{equation*}
 I := \big\{ t \in [0,T] \mid u(\gamma(s)) = u(z_0) \ \text{for every} \ s \in[0,t] \big\}, \qquad \overline t := \sup I.
\end{equation*}
Clearly, $I \ne \emptyset$ as $0 \in I$. Moreover $I$ is closed, because of the continuity of $u$ and $\gamma$, then $\overline t \in I$. We now prove by contradiction that $\overline t = T$. Indeed, if $\overline t < T$, then we let $z_1 := \gamma(\overline t)$, and we note that $z_1 \in \Omega$, $u(z_1) = \max_\Omega u$. Moreover, there exist $r_1>0$ such that $\overline{\Omega_{r_1}(z_1)} \subset \O$ and, by condition [H.4], a positive $s_1$ such that $ \gamma(\overline t + s) \in \Omega_{r_1}(z_1)$ for every $s \in [0, s_1[$. 
As a consequence of \eqref{eq-claim-smp} we obtain $u(\gamma(\overline t + s)) = u(z_1) = u(z_0)$ for every $s \in [0, s_1[$, and this contradicts the assumption $\overline t < T$. This proves that $u(z) = u(z_0)$ for every $z \in \AS  ( \O )$.  By the continuity of $u$ we conclude that $u(z) = u(z_0)$ for every $z \in \overline{\AS  ( \O )}$. Eventually, since $u$ is constant in $\overline{\AS  ( \O )}$, we conclude that $f(z) = \L u (z) = u(z_0) \, c(z)$ for every $z \in \overline{\AS  ( \O )}$. 

% \medskip 
% 
% \bsr
% XXXXXXXX \bsrr Sarbbe bello rimuovere la condizione $\div_{\mathbb{G}} \, b > 0$ \esrr XXXXXXXX
% \esr
% 
% \medskip 

We finally remark that the condition on the sign of $u(z_0)$ was used only to guarentee that $u(z_0) c(z)$ has the required sign. If we assume $c =0$, the needed condition is always satisfied, and we conclude that $f=0$. 
\end{proof}

% The proof of \eqref{eq-claim2-smp} is a consequence of Lemma \ref{lem-localestimate*}. It is not restrictive to assume that  $r_1 \le r^*$, then it is sufficient to show that there exists a positive $s_1$ such that 
% \begin{equation} \label{eq-claim4-smp}
%  \gamma(\overline t + s) \in \Omega^*_{r_1}(z_1) \quad \text{for every} \quad  s \in [0, s_1[. 
%  %\big\{ z \in \R^{N+1} \mid Z\left( z; z_1\right) \ge \frac{C_N}{2 r_1^N} \big\} 
% \end{equation}
% Recall the definition of $\gamma(\overline t + s) = (x(\overline t + s), t(\overline t + s))$. We have $\gamma(\overline t) = z_1 = (x_1, t_1), t(\overline t + s) = t_1-s$ and, for every positive $s$ 
% \begin{equation*}
% \begin{split}
%  |x(s + \overline t) - x_1 | & = \left| \int_0^s \dot x(\overline t + \sigma) d \sigma \right| \le \int_0^s \left| \dot x(\overline t + \sigma) \right| d \sigma \\
%  & \le \left( \int_0^s \left| \dot x(\overline t + \sigma) \right|^2 d\sigma \right)^{1/2} s^{1/2} \le \|\dot x \|_{L^2([0,T])} \sqrt{s},
% \end{split}
% \end{equation*}
% then
% \begin{equation*}
% \langle A^{-1}(z_1 )(x(\overline t + s)-x_1), x(\overline t + s)-x_1 \rangle \le 
% s \cdot \| A^{-1}(z_1)\| \cdot \|\dot x \|_{L^2([0,T])}^2.
% \end{equation*}
% By using the above inequality in \eqref{eq-Omega_r*-exp} we see that there exists a positive constant $s_1$ such that \eqref{eq-claim4-smp} holds. 

\setcounter{equation}{0} 

\section{Examples}\label{SecExamples}

In this section we list several examples of well-known and important operators verifying the hypotheses of our results, that basically rely on a suitable group structure and on the existence and the properties of the fundamental solution. As said in the Introduction, we warn the reader that the natural dilation operator used in the regularity theory is not the same we use here to prove mean value formulas, and in each example both are described. The homogeneous Lie groups in the examples we are going to present are the Carnot groups and the Kolmogorov groups. We refer to the monograph \cite{LibroBLU} for a detailed treatment of the subject of Carnot groups, and to the survey article \cite{AnceschiPolidoro} for the homogeneous Kolmogorov groups. Next, we check that our hypotheses on the fundamental solutions hold true, recalling in each case the relevant known results on existence and estimates. Concerning the problem of its existence, we recall that Levi's parametrix method provides us with the existence of a fundamental solution for non-divergence operators with H\"older continuous coefficients in the setting of uniformly elliptic and parabolic operators. This method has been extended to the case of \emph{heat operators on Carnot groups} whose prototypical case is the Heisenberg group in Example \ref{ex-1} and to \emph{degenerate Kolmogorov operators,} described in Example \ref{ex-2}, which are two large classes of operators to which our results apply. 

\begin{example} \label{ex-1} {\sc Sublaplacian and parabolic operator on the Heisenberg group.} 
{\rm The Heisenberg group $\mathbb{H}^n = (\R^{2n+1}, \cdot)$ is defined by the composition law
\begin{equation} \label{e-Heisenberg}
 (x,y,s) \cdot (x',y',s') = \Big( x + x', y + y', s + s' + 2 \sum_{j=1}^n (x_j' y_j - x_j y_j') \Big).
\end{equation}
The dilation 
\begin{equation} \label{e-dil-Heisenberg}
 \tilde\d_\lambda (x,y,s) = \big( \lambda x, \lambda y, \lambda^2 s  \big)
\end{equation}
is an automorphism of $\mathbb{H}^n$ and induces a direct sum decomposition on $\R^{2n+1}$
\begin{equation}\label{e-oplus-Heisenberg}
    \R^{2n+1}=V_{1}\oplus V_{2},
\end{equation}
where $V_1 = \{ (x,y,0) \mid x,y, \in \R^n \}$ and $V_2 = \{ = (0,0,s) \mid s \in \R \}$. The vector fields
\begin{equation}\label{e-vectorfields-Heisenberg}
    X_j = \partial_{x_j} + 2 y_j \partial_s, \quad Y_j = \partial_{y_j} - 2 x_j \partial_s, \quad j=1, \dots,n,
\end{equation}
are left-invariant on $(\mathbb{H}^n, \cdot)$. Moreover 
\begin{equation*}
    [X_j, Y_j] = -4 \partial_s, \quad j=1, \dots,n,
\end{equation*}
while 
\begin{equation*}
    [X_j, X_k] = [Y_j, Y_k] = [X_j, Y_k] = 0, \quad j,k=1, \dots,n.
\end{equation*}
In particular, $X_1, \dots, X_n, Y_1, \dots, Y_n$ evaluated at $(x,y,s) = (0,0,0)$ is a basis of $V_1$ and the Lie algebra generated
by $X_1, \dots, X_n, Y_1, \dots, Y_n$, evaluated at any point of $\R^{2n+1}$, agrees with $\R^{2n+1}$. The homogeneous dimension of the Heisenberg group is $\mathcal{Q}_H = 2n+2$. The differential operator
\begin{equation}\label{e-sublaplacian-Heisenberg}
    \Delta_{\mathbb{H}^n} := \sum_{j=1}^n \big(X_j^2 + Y_j^2\big),
\end{equation}
is said \emph{sub-Laplacian} on $\mathbb{H}^n$.

The homogeneous Lie group $\mathbb{G} = \left( \R^{2n+2}, \circ, (\d_\lambda)_{\lambda>0} \right)$ relevant to the heat operator on the Heisenberg group,
\begin{equation}\label{e-heat-Heisenberg}
    \Heat_0 := \sum_{j=1}^n \big(X_j^2 + Y_j^2\big) - \partial_t,
\end{equation}
is defined by the composition 
\begin{equation} \label{e-heat-Group-Heisenberg}
 (x,y,s,t) \circ (x',y',s',t') = \Big( (x,y,s) \cdot (x'y',s'), t+ t' \Big), 
\end{equation}
and by the dilation 
\begin{equation} \label{e-dil-heat-Heisenberg-1}
 \d_\lambda (x,y,s,t) = \big( \tilde\d_\lambda (x,y,s), \lambda t  \big).
\end{equation}  
Here $W_1 = \{ (x,y,0,t) \mid x,y, \in \R^n, t \in \R \}$ and $W_2 = \{ = (0,0,s,0) \mid s \in \R \}$, and the homogeneous dimension is $\mathcal{Q} = \mathcal{Q}_H + 1$.
%, being $\mathcal{Q}_H = 2 n + 2$ the homogeneous dimension of the Heisenber group $\mathbb H^n$.
Note that the \emph{parabolic scaling} 
\begin{equation} \label{e-dil-heat-Heisenberg}
 \hat \d_\lambda (x,y,s,t) = \big( \tilde\d_\lambda (x,y,s), \lambda^2 t  \big)
\end{equation} 
is commonly used in the regularity theory for the solutions to $\Heat_0 u = f$. In this framework, the vector fields $X_1, \dots, X_n, Y_1, \dots, Y_n$ are homogeneous of degree $1$ with respect to the dilation $(\d_\lambda)_{\lambda>0}$, while the derivative $\partial_t$ is homogeneous of degree $2$, as usual in the case of parabolic operators. The homogeneous dimension of the group, with respect to $\big( \hat \d_\lambda \big)_{\lambda >0}$, is $\mathcal{Q}_H + 2$, because of the different role played by the time variable $t$. 
\hfill $\square$
}\end{example}

As said above, the Heisenberg group is the prototype of Carnot groups $\mathbb{C} = (\R^{N}, \cdot, \tilde \delta_\lambda)$. The Lie group $\mathbb{G} = \left( \R^{N+1}, \circ, (\d_\lambda)_{\lambda>0} \right)$ relevant to the heat operator on a Carnot group is defined by the operations 
\begin{equation} \label{e-heat-Group-Carnot}
 (x,t) \circ (x',t') = \big( x \cdot x', t+ t' \big), \qquad  
 \d_\lambda (x,t) = \big( \tilde\d_\lambda x, \lambda t  \big),
\end{equation}
being
\begin{equation} \label{e-tildedelta}
  \tilde\delta_\lambda = \diag ( \lambda \I_{m_1} , \lambda^2 \I_{m_2}, \ldots, \lambda^{\mu} \I_{m_\mu}).
\end{equation}
The homogeneous dimension of the Carnot group $\mathbb{C}$ is $\mathcal{Q}_C = m_{1} + 2 m_{2} + \dots + \mu m_{\mu}$, and the homogeneous dimension of $\mathbb{G}$ defined in \eqref{e-Q} is $\mathcal{Q}_C+1$. 

If we consider the direct sum decomposition \eqref{e-Oplus-bis}, we let $m = m_1$ and we choose a family of left-invariant vector fields $X_1, \dots ,X_m$ that form a basis of $W_1$, with the property that $X_1(0), \dots X_m(0)$ are orthonormal. The sub-Laplacian on $\mathbb{C}$ and the heat operator are then defined as
\begin{equation}\label{e-sublaplacian-Carnot}
    \Delta_{\mathbb{C}} := \sum_{j=1}^m X_j^2, \qquad \Heat_0 = \Delta_{\mathbb{C}} - \partial_t.
\end{equation}
Also in this case, the parabolic dilation
\begin{equation} \label{e-dil-heat-Carnot}
 \hat \d_\lambda (x,t) = \big( \tilde\d_\lambda (x), \lambda^2 t  \big)
\end{equation} 
is useful in the regularity theory relevant to $\Heat_0$, and $\mathcal{Q}_C+2$ is the homogeneous dimension of $\mathbb{G}$ with respect to $\big( \hat \d_\lambda \big)_{\lambda >0}$.

We next quote an existence result proved by Bonfiglioli, Lanconelli and Uguzzoni \cite{BLU2}, which applies to parabolic and elliptic operators on Carnot groups. In the next statement $d_\infty(x,y)$ denotes the distance of $x$ and $\xi \in \R^N$, which agrees with $d_\infty((\xi,0),(x,0))$. Note that the regularity assumption made in \cite{BLU2} is written in terms of the parabolic distance. Indeed, in \cite{BLU2} the following condition 
\begin{equation*} %\label{eq-C-alpha-par}
 |u(x,t)-u(\xi,\tau)| \le M_1 \big(d_\infty (x,\xi) + |t-\tau|^{1/2}\big)^\beta, \quad \text{for every} \ (x,t),(\xi,\tau) \in \Omega,
\end{equation*}
is required on the coefficients of the operator. However, this condition follows from \eqref{eq-C-alpha} if we choose $\alpha = \beta / 2$, since the coefficients are bounded.

\begin{theorem} \label{th-Gamma-Heat} {\rm (Theorem 1.2 in \cite{BLU2})}
Let $X_1, \dots X_m$ be vector fields that satisfy the H\"ormander condition {\rm [H.1]} and that generate a Carnot group $\mathbb{C} = (\R^{N}, \cdot, \tilde \delta_\lambda)$, and let $\mathcal{Q}_C$ be its homogeneous dimension. Consider the differential operator
\begin{equation}\label{eq-Heat}
     \Heat u : = \sum_{i,j=1}^m X_i \left( a_{ij} X_j u \right) + 2 \sum_{i,j=1}^m X_i a_{ij} X_j u + 
     \sum_{i,j=1}^m X_i X_j a_{ij} u - \partial_t u,
\end{equation}
where $A = (a_{ij})_{i,j=1, \dots,m}$ is a symmetric matrix satisfying the condition \eqref{elliptic} for some constant $\Lambda \ge 1$. Suppose that, for every $i, j = 1, \dots, m$, the coefficients $a_{ij}$ and their derivatives $X_i a_{ij}$ and $X_i X_j a_{ij}$ belong to the space $C_{\mathbb G}^{\alpha}(\R^{N+1})$ for some $\alpha \in ]0,1]$. Then there exists a fundamental solution $\G^*$ of the adjoint operator 
\begin{equation*} %\label{eq-Heat}
     \Heat^* = \sum_{i,j=1}^m a_{ij} X_i X_j + \partial_t. 
\end{equation*}
Moreover, $\G^*$ satisfies the following estimates: for every positive $T$ there exist two positive constants $M$ only depending on $\Heat$, and $C(T)$, also depending on $T$, such that 
\begin{equation}\label{eq-H-bds}
 0 \le \G^*((\x,\t),(x,t)) \le \frac{C(T)}{(t-\t)^{\mathcal{Q}_C/2}} 
 \exp \left( -\frac{d_\infty^2(\x,x)}{M(t-\t)} \right), 
\end{equation}
%for $i, j = 1 \dots, m$ and 
for every $(\x,\t),(x,t) \in \R^{N+1}$ with $0 < t - \t \le T$. 
Moreover, there exist $M_0, K_0, T_0>0$ such that
\begin{equation}\label{eq-H-lbds}
 \G^*((\x,\t),(x,t)) \ge \frac{c(T_0)}{(t-\t)^{\mathcal{Q}_C/2}} \exp \left( -M_0 \frac{d_\infty^2(\x,x)}{t-\t} \right), 
\end{equation}
for every $(\x,\t),(x,t) \in \R^{N+1}$ with $0 < t - \t \le T_0$ such that 
$M_0 \frac{d_\infty^2(\x,x)}{t-\t} \le - \log \left( K_0 (t-\t)^{\mathcal{Q}_C/2}\right)$.
\end{theorem}

\begin{remark}  \label{rem-Gamma-Heat} {\rm Even though the inequality \eqref{eq-H-lbds} is not explicitly written in \cite{BLU2}, it is a direct consequence of the inequalities (2.2) and (2.15) therein. Moreover, bounds analogous to \eqref{eq-H-bds} hold for the derivatives $\left| X^*_j \G^*((\x,\t),(x,t)) \right|, \left| X^*_i X^*_j \G^*((\x,\t),(x,t)) \right|, i,j = 1, \dots m$ and for $\left| \partial_\tau \G^*((\x,\t),(x,t)) \right|$. We expect that the assumptions on the coefficients of the lower order terms of the operator $\Heat$ can be relaxed.}
\end{remark}

The following result ensures that the mean value formulas stated in Theorem \ref{th-1} hold for every operator $\Heat$ satisfying the assumptions of Theorem \ref{th-Gamma-Heat}.

\begin{proposition} \label{prop-Gamma-Heat-1} Under the assumption of Theorem \ref{th-Gamma-Heat}, there exists a positive $r_0$ such that \eqref{eq-claim-1} and \eqref{eq-claim-2} hold for every $z_0 \in \R^{N+1}$. 
\end{proposition}

\begin{proof}
Fix a positive $T$ and let $M$ and $C(T)$ be the constants appearing in the bound \eqref{eq-H-bds}. If we choose $r_0$ such that $r_0 \le T^{\mathcal{Q}_C/2}/C(T)$, we find
\begin{equation*}
\begin{split}
 \Omega_r(x_0,t_0) & \subset \left\{ (x,t) \in \R^{N} \times ]t_0 - T, t_0[ \mid d_\infty^2(x,x_0) < M (t_0-t) 
 \log \left( \tfrac{C(T) r}{(t_0-t)^{\mathcal{Q}_C/2}}\right) \right\}, \\
 \mathcal{I}_{r, \varepsilon}(x_0,t_0) & \subset \left\{ x \in \R^{N} \mid d_\infty^2(x,x_0) < M \varepsilon 
 \log \left( \tfrac{C(T) r}{\varepsilon^{\mathcal{Q}_C/2}}\right) \right\},
\end{split}
\end{equation*}
for every $(x_0,t_0) \in \R^{N+1}, r \in ]0,r_0]$ and $\varepsilon \in ]0,T]$. This proves \eqref{eq-claim-1} and the first assertion in \eqref{eq-claim-2}. On the other hand, from the inequality \eqref{eq-H-lbds} it follows that 
\begin{equation*}
 \mathcal{I}_{r, \varepsilon}(x_0,t_0) \supset \left\{ x \in \R^{N} \mid d_\infty^2(x,x_0) < \tfrac{\varepsilon}{M_0} 
 \log \left( \tfrac{c(T_0) r}{\varepsilon^{\mathcal{Q}_C/2}}\right) \right\},
\end{equation*}
for every $\varepsilon \in ]0,T_0]$, then
\begin{equation*}
\begin{split}
0 < \int_{\R^N \backslash \mathcal{I}_{r, \varepsilon} (x_0,t_0)} & \Gamma^*(x,t_0 - \varepsilon; x_0,t_0) d x \\
& \le \frac{C(T)}{\varepsilon^{\mathcal{Q}_C/2}} \int_{\left\{ x \in \R^{N} \mid d_\infty^2(x,x_0) \ge \tfrac{\varepsilon}{M_0} 
 \log \left( \tfrac{c(T_0) r}{\varepsilon^{\mathcal{Q}_C/2}}\right) \right\}} 
 \exp \left( -\frac{d_\infty^2(x,x_0)}{M \varepsilon} \right) dx \\
 & = C(T) \int_{\left\{ y \in \R^{N} \mid d_\infty^2(y,0) \ge \tfrac{1}{M_0} 
 \log \left( \tfrac{c(T_0) r}{\varepsilon^{\mathcal{Q}_C/2}}\right) \right\}} 
 \exp \left( -\frac{d_\infty^2(y,0)}{M } \right) dy.
 \end{split} 
\end{equation*}
Note that the quantities in the last two lines agree because of the change of variable $y = \delta_\lambda (x_0^{-1} \circ x)$ with $\lambda = \frac{1}{\sqrt{\varepsilon}}$ and the property \eqref{e-dist-INV} of the distance $d_\infty$. The last integral vanishes as $\varepsilon \to 0$ and this concludes the proof. 
\end{proof}

\begin{proposition} \label{prop-Gamma-Heat-2} Under the assumption of Theorem \ref{th-Gamma-Heat}, the operator $\Heat$ defined in \eqref{eq-Heat} satisfies condition {\rm [H.4]}. 
\end{proposition}
\begin{proof}
Consider the set $\Omega_r(z)$ for some $z = (x,t) \in \R^{N+1}$, and $r>0$, and denote $\g(s) = (x(s), t-s)$. Let $s_1$ be any positive constant such that $K_0 s_1^{\mathcal{Q}_C/2} < 1$, where $K_0$ is as in the statement of Theorem \ref{th-Gamma-Heat}. We claim that there exists $s_2 \in ]0,s_1]$ such that 
\begin{equation} \label{eq-claim-Heat_smp}
  \frac{d_\infty^2(x(s),x)}{s} \le - \frac{1}{M_0}\log\left( K_0 s_1^{\mathcal{Q}_C/2} \right), \quad \text{for every} \ s \in ]0,s_2], 
\end{equation}
where $M_0, K_0$ are the constants appearing in \eqref{eq-H-lbds}. As a consequence, \eqref{eq-H-lbds} is satisfied for every $(\xi, \tau) = (x(s), t-s) = \gamma(s)$ with $0 < s < s_2 \wedge T_0$, so that
\begin{equation*} %\label{eq-H-lbds-2}
 \G^*(\gamma(s),(x,t)) \ge K_0c(T_0)\left( \frac{s_1}{s}\right)^{\mathcal{Q}_C/2} > \frac{1}{r}
\end{equation*}
for every $s \in ]0, s_0[$ if we set $s_0 = s_2 \wedge T_0 \wedge s_1 (r K_0c(T_0))^{2/\mathcal{Q}_C}$. This means that $\gamma(s) \in \Omega_r(z)$ for every $s \in ]0, s_0[$, and the proof of {\rm [H.4]} is concluded.

In order to prove that there exists a positive $s_2$ such that our claim \eqref{eq-claim-Heat_smp} holds, we recall that the \emph{Carnot-Carath\'{e}odory distance} between two points $x,y$ belonging to the Carnot group $\mathbb{C} = (\R^{N}, \cdot, \tilde \delta_\lambda)$ can be defined as
\begin{equation}\label{DefCCdist}
 d_{cc}(y,x) = \inf \int_0^T |\omega(\tau)| d \tau
\end{equation}
where the infimum is taken among all the absolutely continuous paths $x: [0,T] \to \R^N$ such that 
\begin{equation} \label{eq-path-cc}
 x(0)=x, \quad x(T) = y, \quad \text{and} \quad \dot x(\tau) = \sum_{j=1}^m \omega_j X_j (x(\tau)).
\end{equation}
By H\"older's inequality, we have 
\begin{equation*}
 \int_0^T |\omega(\tau)| d\tau \le \sqrt{T} \bigg( \int_0^T |\omega(\tau)|^2 d\tau \bigg)^{1/2},
\end{equation*}
then, if we consider the path $x = x(s)$ in \eqref{eq-claim-Heat_smp}, we find
\begin{equation*}
\frac{1}{s} \bigg(\int_0^s |\omega(\tau)| d\tau  \bigg)^{2} \le \int_0^s |\omega(\tau)|^2 d\tau \to 0, 
\quad \text{as} \quad s \to 0,
\end{equation*}
because we assume that $\omega \in L^2([0,T], \R^m)$. Then also
\begin{equation*}
 \frac{d_{cc}^2(x(s),x)}{s} \to 0 \quad \text{as} \quad s \to 0,
\end{equation*}
and the claim \eqref{eq-claim-Heat_smp} follows from the fact that $d_{\infty}$ is equivalent to $d_{cc}$ in all Carnot groups.
\end{proof}

\begin{example} \label{ex-2} {\sc Degenerate Kolmogorov operators.}
{\rm The simplest degenerate Kolmogorov operator defined for $(x,y,t) \in \R^n\times \R^n \times \R$ is
\begin{equation}\label{e-Kolmogorov}
    \K_0 := \sum_{j=1}^n \partial_{x_j}^2 + \sum_{j=1}^n x_j \partial_{y_j} - \partial_t
\end{equation}
and can be written in the form \eqref{e0} with $m=n$ and
\begin{equation*}
    X_j = \partial_{x_j}, \quad j=1, \dots, n, \quad \text{and} \quad X_{n+1} = \sum_{j=1}^n x_j \partial_{y_j} - \partial_t.
\end{equation*}
Note that
\begin{equation*}
    [X_j, X_{n+1}] = \partial_{y_j}, \quad j=1, \dots,n,
\end{equation*}
and H\"ormander condition [H.1] is satisfied. The homogeneous Lie group $\mathbb{K} = \left( \R^{2n+1}, \circ, (\d_\lambda)_{\lambda>0} \right)$ relevant to this operator is defined by the composition rule 
\begin{equation} \label{e-heat-Group-Kolmogorov}
 (x,y,t) \circ (x'y',t') = ( x+ x', y + y' - t' x, t + t'), 
\end{equation}
and by the dilation 
\begin{equation} \label{e-dil-Kolmogorov}
 \d_\lambda (x,y,t) = ( \lambda x, \lambda^2 y, \lambda t ),
\end{equation}  
which is an automorphism of $\mathbb{K}$ and induces a direct sum decomposition on $\R^{2n+1}$
\begin{equation}\label{e-oplus-Kolmogorov}
    \R^{2n+1}=W_{1}\oplus W_{2},
\end{equation}
where $W_1 = \{ (x,0,t) \mid x \in \R^n, t \in \R \}$ and $W_2 = \{ (0,y,0) \mid y \in \R^n \}$. 

The parabolic dilation on $\mathbb{K}$
\begin{equation} \label{e-dil-heat-Kolmogorov}
 \hat \d_\lambda (x,y,t) = ( \lambda x, \lambda^3 y, \lambda^2 t ),
\end{equation}  
induces a different direct sum decomposition on $\R^{2n+1}$
\begin{equation}\label{e-oplus-heat-Kolmogorov}
    \R^{2n+1}=V_{1}\oplus V_{2} \oplus V_{3},
\end{equation}
where $V_1 = \{ (x,0,0) \mid x \in \R^n \}, V_2 = \{ = (0,0,t) \mid t \in \R \}$ and $V_3 = \{ (0,y,0) \mid y \in \R^n \}$. 
Note that the homogeneous dimension of the group $\mathbb K$ with respect to the dilation \eqref{e-dil-Kolmogorov} is $\mathcal{Q} = 3 n + 1$, while its homogeneous dimension with respect to the dilation \eqref{e-dil-heat-Kolmogorov} is $4 n + 2$.
\hfill $\square$
}\end{example}

The above example arises in stochastic analysis and in its several applications, we refer to the article 
\cite{AnceschiPolidoro} for a survey of known results and for a recent bibliography on this subject. The Example \ref{ex-2} is also the proptotype of a more general family of degenerate Kolmogorov operators. Indeed our main results apply to differential operators of the form
\begin{equation} \label{e-Kolmogorov-gen}
    \K_0 := \sum_{i,j=1}^m a_{ij} \partial_{x_i x_j} + \sum \limits_{i,j=1}^N b_{ij} x_j \partial_{x_i} - \partial_t 
\end{equation}
under the assumption that the matrix $A = \left( a_{ij} \right)_{i,j=1, \dots, m}$ has real constant entries, is symmetric and strictly positive in $\R^m$, while $B = \left( b_{ij} \right)_{i,j=1, \dots, N}$ has real constant entries and takes the following form, where we agree to let $m_0= m$ in order to have a simple and consistent notation:  
\begin{equation} \label{B}
   B = \begin{pmatrix}
       \OO &  \OO & \ldots & \OO & \OO  \\
       B_1 & \OO &  \ldots & \OO & \OO \\
	   \OO & B_{2}  & \ldots & \OO & \OO \\
	   \vdots & \vdots  & \ddots & \vdots & \vdots  \\
	   \OO & \OO  & \ldots & B_{\k} & \OO 
	 \end{pmatrix} .
\end{equation}
Here every block $B_j$ is a $m_{j} \times m_{j-1}$ matrix of rank $m_j$ with $j = 1, 2, \ldots, \k$. 
Moreover, the $m_j$s are positive integers such that 
\begin{equation} \label{e-m-cond}
    m_0 \ge m_1 \ge \ldots \ge m_\k \ge 1, \quad {\rm and} \quad m_0 + m_1 + \ldots + m_\k = N
\end{equation}
and all the entries of the blocks denoted by $\OO$ are zeros. Note that a plain change of variable allows us to write the second order part of the operator $\K_0$ as $\sum_{j=1}^m \partial_{x_j}^2$, without modifying the stucture of the matrix $B$. Then the operator $\K_0$ in \eqref{e-Kolmogorov-gen} can be written in the form \eqref{e1} if we set $b=0, c = 0$
\begin{equation*} %\label{e-X0-kolmo}
    X_0 = \sum \limits_{i,j=1}^N b_{ij} x_j \partial_{x_i}, \qquad 
    X_j = \partial_{x_j}, \ \text{for} \ j=1, \dots, m.
\end{equation*}
In \cite{LanconelliPolidoro} it has been proved that the operator $\K_0$ in \eqref{e-Kolmogorov-gen} is invariant with respect to the following Lie group structure
\begin{equation} \label{e-circ-K}
     (x,t) \circ (\x, \t) = ( \x + E(\t) x, \hspace{1mm} t + \t), 
    \qquad E(\t) = \exp (- \t B ).
\end{equation}
Moreover $\K_0$ is invariant with respect to the dilation $\delta_\lambda$ defined by the matrix
\begin{equation} \label{e-dil-K1}
   		 \delta_\lambda = \diag ( \lambda \I_{m_0} , \lambda^2 \I_{m_1}, \ldots, \lambda^{\k+1} \I_{m_\k}, \lambda ),
\end{equation}
where $\I_{m_j}$ denotes the identity matrix in $\R^{m_j}$. A direct computation shows that 
\begin{equation*}
    \mathbb{K} = \left( \R^{N+1}, \circ, (\d_\lambda)_{\lambda>0} \right) 
\end{equation*}
is a homogeneous Lie group whose homogeneous dimension is $\mathcal{Q} = m_0+1 +2 m_1 + \dots + (\kappa + 1) m_\kappa$ and that the vector fields $X_1, \dots, X_m, X_{m+1}$ satify the conditions [H.1] and [H.2].

It is remarkable that, if we consider the parabolic dilation
\begin{equation} \label{e-dil-K2}
   		 \tilde \delta_\lambda = \diag ( \lambda \I_{m_0} , \lambda^3 \I_{m_1}, \ldots, \lambda^{2\k+1} \I_{m_\k}, \lambda^2 ),
\end{equation}
then we find the direct sum decomposition $\GG = V_{1}\oplus\dots\oplus V_{2\kappa +1}$ with
\begin{equation*}
\begin{split}
     & V_1 =  \text{span} \big\{X_1, \dots X_m\big\}, 
     \qquad V_2 = \text{span}\big\{ X_{m+1} \big\}, \qquad V_{2j} = \big\{ 0 \big\}, \quad \text{for} \ j=2, \dots, \kappa,\\
    & V_{2j+1} =  \text{span} \big\{[X_i, X_{m+1}]\mid X_i \in V_{2j-1} \big\}, \qquad j=i, \dots, \kappa.
\end{split}
\end{equation*}
This decomposition differs from \eqref{e-Oplus-bis} in that all the spaces $W_1, \dots, W_\mu$ appearing in \eqref{e-Oplus-bis} are non-trivial. The integer $\mathcal{Q}_P = m_{0} + 3 m_{1} + \dots + (2 \k +1) m_{\k} + 2$ is usually referred to as \emph. 

Let us consider the variable coefficient degenerate Kolmogorov operator
 \begin{equation}\label{e-K}
     \K u : = \sum_{i,j=1}^m \partial_{x_i} \left(a_{ij} \partial_{x_j} u \right)  + 
     \sum_{j=1}^m b_{j} \partial_{x_j} u + c u +
     \sum \limits_{i,j=1}^N b_{ij} x_j \partial_{x_i}u - \partial_t u
\end{equation}
where the matrix $A = \left( a_{ij} \right)_{i,j=1, \dots, m}$ is symmetric and satisfies the condition \eqref{elliptic} for some constant $\Lambda \ge 1$, $b_1, \dots, b_m$, and $c$ are bounded continuous functions, and $B = \left( b_{ij} \right)_{i,j=1, \dots, N}$ is as in the Example \ref{ex-2}. Consider also its adjoint operator
 \begin{equation*} %\label{e-K*}
     \K^* v = \sum_{i,j=1}^m \partial_{x_i} \left(a_{ij} \partial_{x_j}v \right) - 
     \sum_{j=1}^m \partial_{x_j} \left(b_{j} v \right) + c v -
     \sum \limits_{i,j=1}^N b_{ij} x_j \partial_{x_i} v + \partial_tv. 
\end{equation*}
The following result has been proved by one of the authors in \cite{Polidoro2} for Kolmogorov operators with no lower order terms, and by Di Francesco and Pascucci in \cite{DiFrancescoPascucci} in the general setting. It requires a further notation. For any positive $\Lambda$ we consider the operator
\begin{equation}\label{e-KL}
     \K_\Lambda u : = \Lambda \sum_{j=1}^m \partial^2_{x_i} u +
     \sum \limits_{i,j=1}^N b_{ij} x_j \partial_{x_i} - \partial_t.
\end{equation}
Representing the matrix $B$ as in \eqref{B}, where $B_j$ is a $m_{j} \times m_{j-1}$ is matrix of rank $m_j$ for $j = 1, 2, \ldots, \k$, and the integers $m_0, \dots, m_\k$ do satisfy \eqref{e-m-cond}, then the fundamental solution of $\K_\Lambda$ is well defined and can be explicitly written as follows. Denote by $J$ the $N \times N$ matrix $J := \diag ( \I_{m_0} , 0, \dots, 0 )$ and define, for every $t \in \R$, 
\begin{equation*} %\label{eq-c}
	E(t) = \text{exp} (-t B ), \qquad C(t) = \int_{0}^{t} E(s) \, J \, E^{T}(s) \, ds.
\end{equation*}
Then the matrix $C(t)$ is non singular for every positive  $t$ and the fundamental solution $\G_\Lambda$ of $\K_\Lambda$ is 
\begin{equation} \label{eq-FSGL}
	\G_\Lambda(x,t; \x,\t) = \frac{(4 \pi \Lambda)^{-\frac{N}{2}}}{\sqrt{\text{det} C(t-\t)}} \hspace{1mm}
            \text{exp} \left( - \tfrac{1}{4\Lambda} \langle C^{-1} (t-\t) (x - E(t-\t) \x), x - E(t-\t) \x \rangle
            \right),
\end{equation}
for $t>\t$, while $\G_\Lambda(x,t; \x,\t) = 0$ whenever $t \le \t$. Moreover, $\G^*_\Lambda(\x,\t; x,t) := \G_\Lambda(x,t; \x,\t)$ is the fundamental solution of $\K_\Lambda^*$. With this notation we have

\begin{theorem} \label{th-Gamma-K} {\rm (Theorem 1.4 in \cite{DiFrancescoPascucci})}
Consider the degenerate Kolmogorov operator $\K$ in \eqref{e-K}. As\-sume that the matrix $A = \left( a_{ij} \right)_{i,j=1, \dots, m}$ is symmetric and satisfies the condition \eqref{elliptic} for some constant $\Lambda \ge 1$, while $B = \left( b_{ij} \right)_{i,j=1, \dots, N}$ has the form \eqref{B}, where $B_j$ is a $m_{j} \times m_{j-1}$ is matrix of rank $m_j$ with $j = 1, 2, \ldots, \k$, and the integers $m_0, \dots, m_\k$ do satisfy \eqref{e-m-cond}. Assume also that, for every $i,j= 1, \dots, m$, the coefficients $a_{ij}, b_j, c, \partial_{x_i}a_{ij}$ and $\partial_{x_j}b_{j}$ are bounded functions belonging to the space $C_{\mathbb G}^{\alpha}(\R^{N+1})$ for some $\alpha \in ]0,1]$. Then there exists a fundamental solution $\G^*$ of the adjoint operator $\K^*$. 
Moreover, $\G^*$ satisfies the following estimates: for every positive $T$ there exist four positive constants $\Lambda^+$ only depending on $\K$, and $\Lambda^-, C^-(T)$ and $C^+(T)$, also depending on $T$, such that 
\begin{equation}\label{eq-K-bds}
%\begin{split}
 C^-(T) \, \G^*_{\Lambda^-}((\x,\t),(x,t)) \le \G^*((\x,\t),(x,t)) \le C^+(T) \, \G^*_{\Lambda^+}((\x,\t),(x,t)), 
%   \\
%   \left| X^*_j \G^*((\x,\t),(x,t)) \right| \le & \frac{C(T)}{(t-\t)^{1/2}} \G^*_\Lambda((\x,\t),(x,t)), \\
%  \left| X^*_i X^*_j \G^*((\x,\t),(x,t)) \right| \le & 
%  \frac{C(T)}{t-\t} \G^*_\Lambda((\x,\t),(x,t)), \\
%  \left| X^*_{m+1} \G^*((\x,\t),(x,t)) \right|\le & 
%  \frac{C(T)}{t-\t} \G^*_\Lambda((\x,\t),(x,t)), %\\
% \end{split}
\end{equation}
%for $i, j = 1 \dots, m$ and 
for every  $(\x,\t),(x,t) \in \R^{N+1}$ with $0 < t-\t \le T$. 
\end{theorem}

The following result ensures that the mean value formulas stated in Theorem \ref{th-1} hold for every operator $\K$ satisfying the assumptions of Theorem \ref{th-Gamma-K}.

\begin{proposition} \label{prop-Gamma-Kolmo-1} If the operator $\K$ defined in \eqref{e-K} satisfies the assumption of Theorem \ref{th-Gamma-K}, then there exists a positive $r_0$ such that \eqref{eq-claim-1} and \eqref{eq-claim-2} hold for every $z_0 \in \R^{N+1}$. 
\end{proposition}

The proof of Proposition \ref{prop-Gamma-Kolmo-1} is analogous to that of Proposition \ref{prop-Gamma-Heat-1} and is omitted. 
We next prove that, under the assumption of Theorem \ref{th-Gamma-K}, the operator $\K$ satisfies  {\rm [H.4]}, then Theorem \ref{Thm-smp} does apply to $\K$. We then find the following strong maximum-minimum principle for degenerate Kolmogorov operators.
% We also remove the assumption $c - \div_{\mathbb{G}}b  < 0$ in this case. 

\begin{proposition} \label{prop-Gamma-Kolmo-2} 
Let $\K$ be the operator defined in \eqref{e-K}, satisfying the assumption of Theorem \ref{th-Gamma-K}, with $c \le 0$ and $c - \div_{\mathbb{G}} \, b < 0$. Let $u$ be a classical solution to $\K u = f$ in an open subset $\O\subset\R^{N+1}$, and let $z_0=(x_0,t_0) \in \O$ be such that $u (z_0) = \max_\Omega u \ge 0$ and $f \ge 0$ in $\Omega$; then 
\begin{equation*}
u(z) = u(z_0) \quad \text{and} \quad f(z) = u(z_0) c(z) \qquad \text{for every} \ z \in \overline {\AS  ( \O )}.
\end{equation*}
The analogous result holds true if $u (z_0) = \min_\Omega u \le 0$ and $f \le 0$ in $\Omega$. 
Moreover, we can drop the assumption on the sign of $u (z_0)$ if $c = 0$. 
\end{proposition}

\begin{proof}
We first prove that $\K$ satisfies  {\rm [H.4]}. Consider the set $\Omega_r(z)$ for some $z = (x,t) \in \R^{N+1}$, and $r>0$, and let $T>0$ be such that $\Omega_r(z) \subset \R^{N} \times ]t- T, t[$. From \eqref{eq-K-bds} it follows that 
\begin{equation*}
 \Omega_r(z) \supset \widehat \Omega_r(z) := \left\{ \z \in \R^{N} \times ]- \infty, t[ : 
 \Gamma_{\Lambda^-}^*(\z;z) > \tfrac{1}{C^-(T) r} \right\}.
\end{equation*}
As in the proof of Proposition \ref{prop-Gamma-Heat-2}, we need to show that there exists a positive $s_2$ such that $\g(s) \in \widehat \Omega_r(z)$ for every $s \in ]0,s_2[$. If we denote  $\g(s) = (x(s), t-s)$, in view of \eqref{eq-FSGL}, this means that
\begin{equation*}
 \langle C^{-1}(s)(x - E(s)x(s)), x - E(s)x(s) \rangle \le C_0 + C_1 \left(\log(r) - \tfrac{\Q_P -2}{2} \log (s) \right)
\end{equation*}
for some positive constants $C_0, C_1$ depending on $\Lambda^-$ and on $T$. Recall that $\Q_P$ is the parabolic homogeneous dimension of $\mathbb{K}$. This fact has been proved in Lemma 3.7 of \cite{ALanconelliPascucciPolidoro}. This concludes the proof of {\rm [H.4]}, then Theorem \ref{Thm-smp} does apply to $\K$.
\end{proof}

\begin{remark} {\rm In the case of degenerate Kolmogorov operators the geometry of the propagation set doesn't agree with the one  relevant to uniformly parabolic operators. Consider for instance the following operator $\K$: 
\begin{equation*}
 \K u(x,y,t) = \partial_x \big( a(x,y,t) \partial_x  u(x,y,t) \big) + x \partial_y u(x,y,t) - \partial_t u(x,y,t)
\end{equation*}
defined for $(x,y,t) \in \Omega = ]-R,R[ \times ]-1,1[ \times ]-1,1[ \subset \R^3$, with $R>0$, and let $z_0 = (0,0,0)$. Then $\AS(\Omega) = \{ (x,y,t) \in \Omega : |y| < - R\, t \}$ (see Fig. 2 below). We recall that in Proposition 4.5 of \cite{CintiNystromPolidoro} it is shown that there exists a non-negative solution $u$ to $\K u =0$ vanishing in the set $\overline{\AS(\Omega)}$ and strictly positive elsewere. Then the minimum principle stated in Proposition \ref{prop-Gamma-Kolmo-2} is sharp.}
\end{remark}

\begin{tikzpicture}
\clip (-4.2,-3.2) rectangle (5,2.2);
\fill[black!7](-2,.5)--(2,-.5)--(1,-3)--(-3,-2);
\fill[black!14](2,-.5)--(3,-2)--(1,-3);
\draw[->,black!50,line width=.7pt](0,-2)--(0,1.4);% 
\node at (.2,1.15) {$t$};
\draw[->,black!50,line width=.7pt](-3,.75)--(4,-1);%
\node at (4,-.8) {$x$};
\draw[->,black!50,line width=.7pt](2,1)--(-2,-1);%
\node at (-2,-.8) {$y$};
\draw[black!60](-3,0)--(1,-1);%
\draw[black!60](-3,-2)--(1,-3);%
\draw[black!60](3,0)--(-1,1);%
\draw[black!60, dashed](3,-2)--(-1,-1);%
\draw[black!60](3,0)--(3,-2);%
\draw[black!60](-3,0)--(-3,-2);%
\draw[black!60](1,-1)--(1,-3);%
\draw[black!60, dashed](-1,1)--(-1,-1);%
\draw[black!60](3,0)--(1,-1);%
\draw[black!60](-3,0)--(-1,1);%
\draw[black!60](3,-2)--(1,-3);%
\draw[black!60, dashed](-3,-2)--(-1,-1);%
\draw[black!40](-2,.5)--(2,-.5);%
\draw[black!40, dashed](0,0)--(0,-2);%
\draw[black!40] (1,-3) -- (2,-.5)-- (3,-2);%
\draw[black!40, dashed](-2,.5)--(-1,-1);%
\draw[black!40](- 2,.5)--(-3,-2);%
\node at (.7,.2) {$(0,0,0)$};
\draw[line width=2pt] (0,0) circle (.2mm);
\end{tikzpicture}

{\sc Fig.2}  - The set $\AS(\Omega)$.

\bigskip

\begin{remark} \label{rem-Gamma-Kolmo} {\rm Theorem \ref{th-Gamma-K} and Propositions \ref{prop-Gamma-Kolmo-1} and \ref{prop-Gamma-Kolmo-2} have been proved in \cite{DiFrancescoPascucci} under a less restrictive assumption. Specifically, the Lie group defined by \eqref{e-circ-K} doesn't need to be homogeneous. The matrix $B$ is assumed to have the following form
\begin{equation} \label{BB}
   B = \begin{pmatrix}
       \ast & \ast & \ldots & \ast & \ast  \\
       B_1 & \ast &  \ldots & \ast & \ast \\
	   \OO & B_{2}  & \ldots & \ast & \ast \\
	   \vdots & \vdots  & \ddots & \vdots & \vdots  \\
	   \OO & \OO  & \ldots & B_{\k} & \ast 
	 \end{pmatrix},
\end{equation}
where every block $B_j$ is a $m_{j} \times m_{j-1}$ matrix of rank $m_j$ with $j = 1, 2, \ldots, \k$, and the blocks denoted by $\ast$ are arbitrary.

Note that, in this case, the identity \eqref{e-adj} still holds for $j= 1, \dots, m$, while $X^*_{m+1} f = - X_{m+1} f - \text{\rm tr} B f$ for suitably smooth test functions $f$.  Moreover, bounds analogous to \eqref{eq-K-bds} hold for $\left| X^*_j \G^*((\x,\t),(x,t)) \right|$, for $j=1, \dots, m+1$, and for $\left| X^*_i X^*_j \G^*((\x,\t),(x,t)) \right|, i,j = 1, \dots m$. 
}\end{remark}

% 	\bibstyle{siam}

\def\cprime{$'$} \def\cprime{$'$} \def\cprime{$'$}
  \def\lfhook#1{\setbox0=\hbox{#1}{\ooalign{\hidewidth
  \lower1.5ex\hbox{'}\hidewidth\crcr\unhbox0}}} \def\cprime{$'$}
  \def\cprime{$'$} \def\cprime{$'$}

\end{document}